\documentclass[11pt]{amsart}
\usepackage{amsmath,amsfonts,amssymb,amscd,amsthm,bm}
\usepackage{mathtools}
\usepackage{mathrsfs}
\usepackage{graphicx}
\usepackage{booktabs}
\usepackage{multirow}
\usepackage{longtable}
\usepackage{tabularx}
\usepackage[square, comma, sort&compress, numbers]{natbib}
\usepackage{tikz}
\usepackage{float}
\usepackage{hyperref}
\usepackage[T1]{fontenc}
\usepackage[utf8]{inputenc}

\allowdisplaybreaks
\newcolumntype{Y}{>{\raggedleft\arraybackslash}X}

\newcommand{\cV}{\mathcal{V}}

\newcommand{\bz}{\mathbb{Z}}

\def\Id{\text{\rm Id\,}}

\def\ker{\text{\rm Ker\,}}

\newcommand{\eqdeg}[1]{#1\mbox{\rm -}\deg}

\newcommand{\amal}[5]{#1\prescript{#2}{}\times_{#3}^{#4}#5}
\newtheorem{theorem}{Theorem}[section]
\newtheorem{proposition}{Proposition}[section]
\newtheorem{lemma}{Lemma}[section]

\newtheorem{definition}{Definition}[section]

\newtheorem{remark}{Remark}[section]

\title{Local and Global Equivariant Bifurcation for Periodic Weyl and Riesz Fractional Equations}
\author{Shi Yu}
\address{Math Department, Odessa College, 201 W University Blvd, Odessa, TX 79764, USA. } 
\email{syu@odessa.edu}
\date{}

\begin{document}

\begin{abstract}
We investigate local and global bifurcation of periodic solutions for two classes of nonlinear fractional differential equations with symmetry. For the one-sided periodic Weyl equation, the nonzero temporal Fourier modes give rise to complex characteristic functions, leading naturally to a two-parameter bifurcation problem. The associated local invariant is expressed in terms of winding numbers and twisted equivariant degree. In contrast, the periodic Riesz equation is governed by real spectral quantities, and its local bifurcation invariant is obtained from the jump of the equivariant degree across isolated critical values. In both settings, the decomposition into spatial isotypical components and temporal Fourier modes determines the critical representations and the possible symmetries of bifurcating solutions. A nonvanishing local invariant yields the existence of nearby nontrivial periodic solutions, while the corresponding global bifurcation theorems describe the continuation of connected solution components away from the trivial branch.
\end{abstract}



\maketitle

\section{Introduction}
Fractional differential equations have become an important tool for describing systems with memory, hereditary effects, anomalous diffusion, and nonlocal interactions; see, for example, \cite{Podlubny1999,Diethelm2010}. They arise in
models from viscoelasticity, diffusion and transport processes, control theory, and other dynamical systems in which the present state depends on its history or on nonlocal interactions. Recent developments continue to emphasize periodic solutions, nonlocal operators, and bifurcation phenomena in fractional and pseudo-differential equations; see, for example, \cite{Sampedro2025,Charkaoui2025,Haacker2026}.

Periodic fractional differential equations provide a natural framework for studying oscillatory behavior in such nonlocal systems. For periodic functions, the fractional operators considered here are naturally described through Fourier series and their associated multipliers
\cite{Ferrari2018,RoncalStinga2014,Sampedro2025}. In particular, recent work has shown the usefulness of Liouville--Weyl-type formulations for exact periodic solutions of fractional-order systems \cite{Haacker2026}. The temporal Fourier modes \(e^{imt}\), \(m\in\bz\), diagonalize the operators considered here, so that their multipliers determine the spectral structure of the linearized equations as well as the temporal symmetries preserved by the problem. This makes the Fourier representation particularly suitable for studying periodic bifurcation in the presence of spatial or internal symmetries.

Equivariant degree provides a natural topological framework for detecting solutions of nonlinear problems with symmetry, since, unlike the ordinary degree, it retains information about the isotropy types of solutions. The general theory and its computational foundations were developed in \cite{AED,SURVEY,Rybicki1994,MarzantowiczPrieto2004} and have been applied to periodic solutions, symmetric Newtonian systems, reversible equations, functional differential equations, Hamiltonian systems, elliptic systems, and delay equations; see, for example, \cite{Dab1,SY1,SY2,Yu3,Golebiewska2018,GolebiewskaRybickiStefaniak2021,Duan2024}. Equivariant degree methods have also been used to study relative periodic solutions and symmetry patterns in vortex and molecular systems \cite{Garcia2019Vortex,GarciaBerezovik2019}. More recently, García-Azpeitia, Ghanem, and Krawcewicz established global bifurcation of nonstationary solutions in symmetric nonlinear wave equations, while Chen, Crane, and Hensley applied equivariant degree to global Hopf bifurcation in symmetric systems with distributed delays \cite{GarciaGhanemKrawcewicz2025,ChenCraneHensley2025}. Yu further extended equivariant degree techniques to random periodic solutions in second-order stochastic systems with symmetry \cite{Yu2026Random}. Related developments include equivariant spectral-flow methods for bifurcation under compact group actions \cite{Izydorek2025}.

The main novelty of this paper is to place these two periodic fractional problems within a common equivariant bifurcation framework while preserving their different spectral mechanisms. For the one-sided Weyl problem, we use winding numbers of complex characteristic functions together with twisted equivariant degree to detect local bifurcation and spatio-temporal symmetries. For the Riesz problem, the local invariant is obtained from the jump of the ordinary equivariant degree across an isolated real critical value. We further derive corresponding global bifurcation alternatives and illustrate the theory by symmetric examples, including numerical periodic solutions for the Weyl system.

Let \(V\) be a finite-dimensional orthogonal representation of a finite group \(\mathcal G\), and let \(x:\mathbb R\to V\) be \(2\pi\)-periodic. Throughout the paper, we assume
\[
0<q<1.
\]
We study two classes of fractional differential equations.

The first is the \emph{one-sided periodic Weyl--Caputo fractional differential equation}
\begin{equation}\label{eq:one-sided-weyl}
D_{W,+}^{q}x(t)=f\bigl(\alpha,\beta,x(t)\bigr),
\qquad
x(t+2\pi)=x(t),
\end{equation}
where \((\alpha,\beta)\in\mathbb R^2\). For a sufficiently regular \(2\pi\)-periodic function \(x\), the one-sided Weyl--Caputo derivative is defined by
\[
D_{W,+}^{q}x(t):=
\frac{1}{\Gamma(1-q)}\int_{-\infty}^{t}\frac{x'(s)}{(t-s)^q}\,ds
=\frac{1}{\Gamma(1-q)}\int_{0}^{\infty}\frac{x'(t-r)}{r^q}\,dr,
\]
where \(\Gamma(z)=\displaystyle\int_{0}^{\infty} r^{z-1}e^{-r}\,dr\) denotes the Euler gamma function for \(z>0\). We consider the following conditions:
\begin{itemize}
\item[(A1)] The mapping \(f:\mathbb{R}^2\times V\to V\) is \(\mathcal G\)-equivariant with respect to \(x\); that is,
\[
f(\alpha,\beta,\mathfrak g x)
=\mathfrak g f(\alpha,\beta,x),
\qquad \mathfrak g\in\mathcal G.
\]

\item[(A2)] The mapping \(f\) is of class \(C^2\) and is odd with respect to \(x\); namely,
\[
f(\alpha,\beta,-x)=-f(\alpha,\beta,x).
\]
\end{itemize}
The second problem is the \emph{periodic Riesz fractional differential
equation}
\[
(-\partial_t^2)^{q/2}x(t)=
g\bigl(\lambda,x(t)\bigr),
\qquad x(t+2\pi)=x(t),
\]
where \(\lambda\in\mathbb R\).
 We consider the following conditions:
\begin{itemize}
\item[(B1)] The mapping \(g:\mathbb{R}\times V\to V\) is \(\mathcal G\)-equivariant with respect to \(x\); that is,
\[
g(\lambda,\mathfrak g x)
=\mathfrak g g(\lambda,x),
\qquad \mathfrak g\in\mathcal G.
\]

\item[(B2)] The mapping \(g\) is of class \(C^2\) and is odd with respect to \(x\); namely,
\[
g(\lambda,-x)=-g(\lambda,x).
\]
\end{itemize}
Since the nonlinearities \(f\) and \(g\) are of class \(C^2\) and are odd with respect to the state variable, it follows that
\[
f(\alpha,\beta,0)=0,
\qquad
g(\lambda,0)=0.
\]
Therefore, \(x\equiv 0\) is a \(2\pi\)-periodic solution of both equations for all admissible parameter values. Hence, each problem possesses a trivial branch of \(2\pi\)-periodic solutions.

For both periodic fractional differential equations, we use a temporal Fourier-mode decomposition. Since the functions
\[
e^{imt},\quad m\in\mathbb{Z},
\]
form a natural orthogonal basis for the space of \(2\pi\)-periodic functions, a periodic solution can be represented as
\[
x(t)=\sum_{m\in\mathbb{Z}}x_m e^{imt},
\]
where \(x_m\in V_{\mathbb{C}}\) denotes the Fourier coefficient associated with the temporal frequency \(m\).

For a Fourier mode \(e^{imt}\), \(m\in\bz, m\neq 0\), by Abel regularization
\[
\int_{0}^{\infty}r^{-q}e^{-imr}\,dr:=
\lim_{\varepsilon\to0^+}\int_{0}^{\infty}r^{-q}e^{-(\varepsilon+im)r}\,dr,
\]
the operator will be 
\[
D_{W,+}^{q}e^{imt}
=
\frac{im\,e^{imt}}{\Gamma(1-q)}
\lim_{\varepsilon\to0^+}\int_{0}^{\infty}r^{-q}e^{-(\varepsilon+im)r}\,dr,
\]

Using
\[
\int_{0}^{\infty}r^{a-1}e^{-zr}\,dr
=\Gamma(a)z^{-a},\qquad\operatorname{Re}(a)>0,
\quad
\operatorname{Re}(z)>0,
\]
with \(a=1-q\) and \(z=\varepsilon+im\), we obtain
\[
\int_{0}^{\infty}
r^{-q}e^{-(\varepsilon+im)r}\,dr
=\Gamma(1-q)(\varepsilon+im)^{q-1}.
\]
Passing to the limit \(\varepsilon\to0^+\) yields
\[
D_{W,+}^{q}e^{imt}
=
(im)^q e^{imt}.
\]
Hence, if
\[
x(t)=\sum_{m\in\bz}\widehat{x}_m e^{imt},
\]
then
\[
D_{W,+}^{q}x(t)=\sum_{m\in\bz}
(im)^q\widehat{x}_m e^{imt}.
\]

For the periodic Riesz fractional operator, the Fourier representation is
\[
(-\partial_t^2)^{q/2}e^{imt}= |m|^q e^{imt}.
\]

The choice of the one-sided Weyl operator in
\eqref{eq:one-sided-weyl} is essential. Let \(D_{W,-}^{q}\) denote the
oppositely oriented one-sided Weyl--Caputo derivative, whose Fourier symbol is
\[
D_{W,-}^{q}e^{imt}
=
(-im)^q e^{imt}.
\]
Since
\[
(im)^q+(-im)^q
=
2\cos\left(\frac{\pi q}{2}\right)|m|^q,
\]
the symmetric Weyl operator satisfies
\[
D_{W,+}^{q}+D_{W,-}^{q}
=
2\cos\left(\frac{\pi q}{2}\right)
(-\partial_t^2)^{q/2}.
\]
Since \(0<q<1\), we have \(\cos\left(\frac{\pi q}{2}\right)>0\). Thus, the symmetric Weyl operator is a nonzero scalar multiple of the periodic Riesz operator. Consequently, the two operators have the same Fourier eigenspaces and the same linear spectral structure. Moreover, they preserve the full temporal \(O(2)\)-symmetry because they commute with time translations and time reversal. Taking also the \(\mathcal G\)-equivariance and the oddness of the nonlinearity into account, the natural symmetry group is
\[
G_R=\mathcal G\times\mathbb{Z}_2\times O(2).
\]
The action of \(G_R\) on the space of \(2\pi\)-periodic functions is defined by
\[
\bigl((\mathfrak g,\varepsilon,e^{i\theta})\cdot x\bigr)(t)=\varepsilon\,\mathfrak g x(t+\theta), \qquad
\mathfrak g\in\mathcal G,\quad
\varepsilon\in\bz_2,\quad
\theta\in[0,2\pi),
\]
for rotations \(e^{i\theta}\in SO(2)\subset O(2)\), and by
\[
\bigl((\mathfrak g,\varepsilon,\tau e^{i\theta})\cdot x\bigr)(t)=\varepsilon\,\mathfrak g x(-t+\theta), \qquad
\mathfrak g\in\mathcal G,\quad
\varepsilon\in\bz_2,\quad
\theta\in[0,2\pi).
\]
for reflections \(\tau e^{i\theta}\in O(2)\setminus SO(2)\). Here, \(\mathcal G\) acts on the spatial or internal state variable, \(\mathbb{Z}_2\) acts by sign reversal through \(x\mapsto -x\), the rotation subgroup \(SO(2)\cong S^1\) acts by time translation, and the reflection component of \(O(2)\) acts by time reversal.

The one-sided Weyl operator has a fundamentally different spectral structure. Its Fourier multipliers \((im)^q\) are generally complex for nonzero temporal modes, and therefore the operator is, in general, non-self-adjoint. Although it commutes with time translations, it does not generally commute with time reversal. Taking into account the \(\mathcal G\)-equivariance and the oddness assumption, the natural symmetry group of \eqref{eq:one-sided-weyl} is
\[
G_W=\mathcal G\times\bz_2\times S^1,
\]
where \(\bz_2=\{1,-1\}\) acts by sign reversal and \(S^1\) acts by temporal translation. More precisely, the action of \(G_W\) on the space of \(2\pi\)-periodic functions is defined by
\[
\bigl((\mathfrak g,\varepsilon,e^{i\theta})\cdot x\bigr)(t)=
\varepsilon\,\mathfrak g x(t+\theta),
\qquad
\mathfrak g\in\mathcal G,\quad
\varepsilon\in\bz_2,\quad
\theta\in[0,2\pi).
\]
Here, \(\mathcal G\) acts on the state variable, \(\bz_2\) represents the odd symmetry \(x\mapsto -x\), and \(S^1\) represents invariance under shifts of the temporal variable.

The use of two parameters in \eqref{eq:one-sided-weyl} reflects the complex structure of the critical Fourier blocks. A singularity of a complex eigenvalue map requires the simultaneous vanishing of its real and imaginary parts. Therefore, a single real parameter is generically insufficient to produce an isolated critical point associated with a nonzero temporal mode. Under an appropriate nondegeneracy condition, the parameters \(\alpha\) and \(\beta\) provide independent variations of these two real components and allow isolated critical parameter pairs to occur. The resulting local
bifurcation invariant is expressed through the twisted \(G_W\)-equivariant degree and the winding numbers of the critical complex eigenvalue maps.

The main purpose of this paper is to develop a unified, symmetry-sensitive analysis of these two periodic fractional problems while preserving the essential distinction between their spectral structures. We construct appropriate functional-analytic settings, reformulate both equations as equivariant compact perturbations of the identity, and decompose their linearizations into spatial \(\mathcal G\)-isotypical components and temporal Fourier modes. This decomposition identifies the critical representations from which nontrivial periodic solutions may bifurcate.

For the one-sided Weyl equation, the critical blocks are generally complex and the temporal symmetry is \(S^1\). We introduce a two-parameter twisted equivariant bifurcation invariant whose coefficients combine winding numbers with the basic degrees of the critical \(\mathcal G\times\bz_2\times S^1\)-representations. A nonzero coefficient detects nontrivial periodic solutions with the corresponding spatial or spatio-temporal symmetry. The global continuation of the bifurcating branches is then studied directly in the two-dimensional parameter space by means of a twisted equivariant Rabinowitz alternative.

For the Riesz equation, the critical Fourier blocks are real and the temporal symmetry is \(O(2)\). We compute the local invariant as the jump of the ordinary equivariant degree across an isolated critical parameter. The Burnside-ring components of this invariant identify possible isotropy types of bifurcating periodic solutions. A corresponding global bifurcation theorem provides the alternative that the connected component of nontrivial solutions is either unbounded or returns to the trivial branch at another critical parameter.

%----------------------------------------------------
\section{Functional Space Setting and Symmetry}
%----------------------------------------------------

Since every \(x\in L^2(S^1;V)\) admits a Fourier expansion
\begin{equation}\label{eq:fourier-expansion}
x(t)=\sum_{m\in\bz}x_m e^{imt},
\qquad
x_m\in V_{\mathbb C},
\end{equation}
where
\[
V_{\mathbb C}:=\mathbb C\otimes_{\mathbb R}V
\]
is the complexification of \(V\). A Fourier series of the form~\eqref{eq:fourier-expansion} represents a real-valued function if and only if
\[
x_{-m}=\overline{x_m},
\qquad
m\in\bz.
\]

For \(s\geq0\), we define the periodic Sobolev space
\[
H_{\mathrm{per}}^s(S^1;V)
:=
\left\{
x\in L^2(S^1;V):
\sum_{m\in\bz}(1+|m|^2)^s|x_m|^2<\infty
\right\},
\]
equipped with the inner product
\[
\langle x,y\rangle_{H_{\mathrm{per}}^s}=\sum_{m\in\bz}
(1+|m|^2)^s
\langle x_m,y_m\rangle_{V_{\mathbb C}}.
\]

and the norm
\[
\|x\|_{H_{\mathrm{per}}^s}^2=
\sum_{m\in\bz}(1+|m|^2)^s|x_m|^2.
\]

We take
\[
X:=H_{\mathrm{per}}^q(S^1;V),
\qquad
Y:=L^2(S^1;V).
\]
\begin{proposition}\label{prop:compact-embedding}
Let \(q>0\), and set
\[
X:=H_{\mathrm{per}}^q(S^1;V),
\qquad
Y:=L^2(S^1;V).
\]
Then \(X\) and \(Y\) are real Hilbert spaces, and the embedding
\[
j:X\hookrightarrow Y
\]
is continuous and compact.
\end{proposition}

\begin{proof}
By Parseval's identity, \(Y\) is isometrically identified with the closed real
subspace
\[
\left\{
(x_m)_{m\in\bz}\in\ell^2(\bz ;V_{\mathbb C}) :
x_{-m}=\overline{x_m}
\right\}.
\]
Likewise, \(X\) is identified with the closed real subspace of the weighted
sequence space
\[
\left\{
(x_m)_{m\in\bz} :
\sum_{m\in\bz}(1+|m|^2)^q|x_m|^2<\infty,
\quad
x_{-m}=\overline{x_m}
\right\}.
\]
Hence both \(X\) and \(Y\) are real Hilbert spaces.

Moreover,
\[
\|x\|_Y^2
=
\sum_{m\in\bz}|x_m|^2
\leq
\sum_{m\in\bz}(1+|m|^2)^q|x_m|^2
=
\|x\|_X^2,
\]
so the embedding \(X\hookrightarrow Y\) is continuous.

To prove compactness, let \((x^{(n)})\) be bounded in \(X\), say
\[
\|x^{(n)}\|_X\leq C.
\]
For \(N\in\mathbb N\), let
\[
P_Nx:=\sum_{|m|\leq N}x_m e^{imt}.
\]
Since \(P_NX\) is finite-dimensional, the sequence
\((P_Nx^{(n)})\) has a convergent subsequence in \(Y\). In addition,
\[
\begin{aligned}
\|(I-P_N)x^{(n)}\|_Y^2
&=
\sum_{|m|>N}|x_m^{(n)}|^2\\
&\leq
(1+N^2)^{-q}
\sum_{|m|>N}(1+|m|^2)^q|x_m^{(n)}|^2\\
&\leq
C^2(1+N^2)^{-q}.
\end{aligned}
\]
The right-hand side tends to zero uniformly in \(n\). A diagonal subsequence
argument therefore yields a subsequence converging in \(Y\). Thus,
\(X\hookrightarrow Y\) is compact.
\end{proof}

The preceding Fourier-truncation argument also implies compact embeddings between periodic Sobolev spaces of different orders. Combining this observation with the Sobolev embedding into continuous functions gives the following result.

\begin{proposition}\label{prop:sobolev-compact-embedding}
Let \(V\) be a finite-dimensional real inner-product space and let
\[
\frac12<q<1.
\]
Then every element of \(H_{\mathrm{per}}^q(S^1;V)\) admits a continuous
\(2\pi\)-periodic representative, and the embedding
\[
H_{\mathrm{per}}^q(S^1;V)
\hookrightarrow
C_{\mathrm{per}}(S^1;V)
\]
is continuous and compact.
\end{proposition}

\begin{proof}
Let
\[
x(t)=\sum_{m\in\bz}x_m e^{imt}.
\]
By the Cauchy--Schwarz inequality,
\[
\sum_{m\in\bz}|x_m|
\leq
\left(
\sum_{m\in\bz}(1+|m|^2)^{-q}
\right)^{1/2}
\|x\|_{H_{\mathrm{per}}^q}.
\]
Since \(q>\frac12\), the series converges. Therefore, the Fourier series of \(x\) converges absolutely and uniformly, and hence defines a continuous \(2\pi\)-periodic representative. This also proves that the embedding into
\(C_{\mathrm{per}}(S^1;V)\) is continuous.

To prove compactness, choose \(r\) such that
\[
\frac12<r<q.
\]
By the same Fourier-truncation argument used in Proposition~\ref{prop:compact-embedding}, the embedding
\[
H_{\mathrm{per}}^q(S^1;V)
\hookrightarrow
H_{\mathrm{per}}^r(S^1;V)
\]
is compact. Since \(r>\frac12\), the embedding
\[
H_{\mathrm{per}}^r(S^1;V)
\hookrightarrow
C_{\mathrm{per}}(S^1;V)
\]
is continuous. The composition of these two embeddings is therefore compact.
\end{proof}
By Proposition~\ref{prop:sobolev-compact-embedding}, if \(\frac12<q<1\), then
\[
H_{\mathrm{per}}^q(S^1;V)
\hookrightarrow
C_{\mathrm{per}}(S^1;V)
\]
continuously and compactly. Hence every element of
\(H_{\mathrm{per}}^q(S^1;V)\) has a continuous \(2\pi\)-periodic representative, and the nonlinearities \(f\) and \(g\) may be evaluated pointwise.

The nonlinearities \(f\) and \(g\) induce the Nemytskii operators
\[
N_f:\mathbb R^2\times Y\longrightarrow Y,
\qquad
N_g:\mathbb R\times Y\longrightarrow Y,
\]
defined, for almost every \(t\in S^1\), by
\begin{equation}\label{eq:nemytskii-f-definition}
\bigl[N_f(\alpha,\beta,u)\bigr](t)
=
f\bigl(\alpha,\beta,u(t)\bigr),
\end{equation}
and
\begin{equation}\label{eq:nemytskii-g-definition}
\bigl[N_g(\lambda,u)\bigr](t)
=
g\bigl(\lambda,u(t)\bigr).
\end{equation}

For \(q>\frac12\), the continuity of \(f\) and \(g\) with respect to the state variable, together with the compactness of \(S^1\), implies that \(N_f(\alpha,\beta,x)\) and \(N_g(\lambda,x)\) are continuous periodic
functions. In particular, they belong to \(L^2(S^1;V)\).

When \(0<q\leq\frac12\), the embedding into
\(C_{\mathrm{per}}(S^1;V)\) is no longer available. Nevertheless,
\[
H_{\mathrm{per}}^q(S^1;V)
\hookrightarrow
L^2(S^1;V)
\]
continuously, and each element of \(H_{\mathrm{per}}^q(S^1;V)\) is defined almost everywhere. Thus, the Nemytskii operators may still be defined almost everywhere by \eqref{eq:nemytskii-f-definition} and
\eqref{eq:nemytskii-g-definition}, provided that the nonlinearities satisfy a suitable growth condition. More precisely, 
\begin{itemize}
    \item[(A3)] For every bounded set \(K\subset\mathbb{R}^2\), there exists a constant \(C_K>0\) such that
    \[
    \lvert f(\alpha,\beta,v)\rvert
    \leq
    C_K\bigl(1+\lvert v\rvert\bigr),
    \qquad
    (\alpha,\beta)\in K,\quad v\in V.
    \]

    \item[(B3)] For every bounded set \(J\subset\mathbb{R}\), there exists a constant \(C_J>0\) such that
    \[
    \lvert g(\lambda,v)\rvert
    \leq
    C_J\bigl(1+\lvert v\rvert\bigr),
    \qquad
    \lambda\in J,\quad v\in V.
    \]
\end{itemize}

\begin{proposition}\label{prop:nemytskii-L2-continuity}
 Assume that \(f\) is continuous and satisfies the condition (A3). Then the Nemytskii operator
\[
N_f:\mathbb R^2\times Y\longrightarrow Y,
\qquad
[N_f(\alpha,\beta,u)](t)
:=
f(\alpha,\beta,u(t)),
\]
defined for almost every \(t\in S^1\), is well defined and continuous. Moreover, it maps bounded subsets of \(\mathbb R^2\times Y\) into bounded
subsets of \(Y\).
\end{proposition}

\begin{proof}
Let \(K\subset\mathbb R^2\) be bounded. For
\((\alpha,\beta)\in K\) and \(u\in Y\), the linear-growth condition gives
\[
|f(\alpha,\beta,u(t))|^2
\leq
2C_K^2\bigl(1+|u(t)|^2\bigr)
\]
for almost every \(t\in S^1\). Hence,
\[
f(\alpha,\beta,u(\cdot))\in L^2(S^1;V),
\]
so \(N_f\) is well defined. Moreover,
\[
\begin{aligned}
\|N_f(\alpha,\beta,u)\|_Y^2
&=
\int_{S^1}|f(\alpha,\beta,u(t))|^2\,dt\\
&\leq
2C_K^2
\bigl(2\pi+\|u\|_Y^2\bigr).
\end{aligned}
\]
Thus, \(N_f\) maps bounded subsets of
\(\mathbb R^2\times Y\) into bounded subsets of \(Y\).

Now let
\[
(\alpha_n,\beta_n,u_n)
\longrightarrow
(\alpha,\beta,u)
\qquad
\text{in }\mathbb R^2\times Y.
\]
Since \(u_n\to u\) in \(L^2(S^1;V)\), every subsequence contains a further
subsequence, still denoted by \((u_n)\), such that
\[
u_n(t)\longrightarrow u(t)
\qquad
\text{for almost every }t\in S^1.
\]
By the continuity of \(f\),
\[
f(\alpha_n,\beta_n,u_n(t))
\longrightarrow
f(\alpha,\beta,u(t))
\]
for almost every \(t\in S^1\).

The sequence \((\alpha_n,\beta_n)\), together with its limit, is contained in a bounded set \(K\subset\mathbb R^2\). Therefore,
\[
\begin{aligned}
&|f(\alpha_n,\beta_n,u_n(t))
-f(\alpha,\beta,u(t))|^2\\
&\qquad\leq
4C_K^2
\bigl(2+|u_n(t)|^2+|u(t)|^2\bigr).
\end{aligned}
\]
Since \(u_n\to u\) in \(L^2(S^1;V)\),
\[
|u_n|^2\longrightarrow |u|^2
\qquad
\text{in }L^1(S^1).
\]
Indeed,
\[
\begin{aligned}
\bigl\||u_n|^2-|u|^2\bigr\|_{L^1}
&\leq
\|u_n-u\|_{L^2}
\bigl(\|u_n\|_{L^2}+\|u\|_{L^2}\bigr)
\longrightarrow0.
\end{aligned}
\]
Hence the family
\[
\left\{
|f(\alpha_n,\beta_n,u_n)
-f(\alpha,\beta,u)|^2
\right\}
\]
is uniformly integrable. By Vitali's convergence theorem,
\[
\int_{S^1}
|f(\alpha_n,\beta_n,u_n(t))
-f(\alpha,\beta,u(t))|^2\,dt
\longrightarrow0.
\]
Therefore,
\[
N_f(\alpha_n,\beta_n,u_n)
\longrightarrow
N_f(\alpha,\beta,u)
\qquad
\text{in }Y.
\]
Thus, \(N_f\) is continuous.
\end{proof}

For \(N_g\), we have the same result. 

%----------------------------------------------------
\section{Fixed-Point Reformulation}
%----------------------------------------------------

In this section, we reformulate the Weyl and Riesz equations as equivariant compact perturbations of the identity. Since both fractional operators vanish on the constant Fourier mode, they are not invertible on the full periodic function space. We therefore introduce shifted operators whose Fourier multipliers are nonzero on every temporal mode.

Consider the one-sided periodic Weyl equation
\begin{equation}\label{eq:weyl-Aw}
D_{W,+}^q x=N_f(\alpha,\beta,x).
\end{equation}

Define the shifted Weyl operator
\[
\mathcal A_W:=D_{W,+}^q+\Id,
\qquad
\mathcal A_W:X\longrightarrow Y.
\]
If
\[
x(t)=\sum_{m\in\bz}x_m e^{imt},
\]
then
\[
\mathcal A_W x(t)=\sum_{m\in\bz}
\bigl(1+(im)^q\bigr)x_m e^{imt}.
\]
For \(0<q<1\), one has
\[
1+(im)^q\neq0,
\qquad
m\in\bz.
\]
Consequently, \(\mathcal A_W\) is an isomorphism from \(X\) onto \(Y\), and its inverse is given by
\[
\mathcal A_W^{-1}y(t)=
\sum_{m\in\bz}
\frac{y_m}{1+(im)^q}e^{imt}.
\]

Equation~\eqref{eq:weyl-Aw} is equivalent to
\[
\mathcal A_Wx=jx+N_f(\alpha,\beta,jx).
\]
It gives
\[
x=\mathcal A_W^{-1}
\bigl(jx+N_f(\alpha,\beta,jx)\bigr).
\]
We therefore define
\[
\Psi_W:\mathbb R^2\times X\longrightarrow X
\]
by
\[
\Psi_W(\alpha,\beta,x):=x-\mathcal A_W^{-1}
\bigl(jx+N_f(\alpha,\beta,jx)\bigr).
\]
Thus,
\[
\Psi_W(\alpha,\beta,x)=0
\]
if and only if \(x\) is a periodic solution of the one-sided Weyl equation.

\begin{proposition}\label{prop:weyl-compact-perturbation}
The mapping
\[
(\alpha,\beta,x)
\longmapsto
\mathcal A_W^{-1}
\bigl(jx+N_f(\alpha,\beta,jx)\bigr)
\]
is compact on bounded subsets of \(\mathbb R^2\times X\). Consequently,
\[
\Psi_W(\alpha,\beta,x)
=
x-\mathcal A_W^{-1}
\bigl(jx+N_f(\alpha,\beta,jx)\bigr)
\]
is a compact perturbation of the identity on \(X\).
\end{proposition}

\begin{proof}
Let
\[
(\alpha_n,\beta_n,x_n)
\]
be a bounded sequence in \(\mathbb R^2\times X\). Since
\((\alpha_n,\beta_n)\) is bounded in the finite-dimensional space
\(\mathbb R^2\), after passing to a subsequence, there exists
\((\alpha,\beta)\in\mathbb R^2\) such that
\[
(\alpha_n,\beta_n)\longrightarrow(\alpha,\beta).
\]

Since the embedding
\[
j:X\hookrightarrow Y
\]
is compact, after passing to a further subsequence, there exists \(y\in Y\)
such that
\[
jx_n\longrightarrow y
\qquad\text{in }Y.
\]
By the continuity of the Nemytskii operator
\[
N_f:\mathbb R^2\times Y\longrightarrow Y,
\]
we obtain
\[
N_f(\alpha_n,\beta_n,jx_n)
\longrightarrow
N_f(\alpha,\beta,y)
\qquad\text{in }Y.
\]
Therefore,
\[
jx_n+N_f(\alpha_n,\beta_n,jx_n)
\longrightarrow
y+N_f(\alpha,\beta,y)
\qquad\text{in }Y.
\]

Since
\[
\mathcal A_W^{-1}:Y\longrightarrow X
\]
is bounded and hence continuous, it follows that
\[
\mathcal A_W^{-1}
\bigl(jx_n+N_f(\alpha_n,\beta_n,jx_n)\bigr)
\longrightarrow
\mathcal A_W^{-1}
\bigl(y+N_f(\alpha,\beta,y)\bigr)
\qquad\text{in }X.
\]
Thus, every bounded sequence in \(\mathbb R^2\times X\) has a subsequence
whose image under
\[
(\alpha,\beta,x)
\longmapsto
\mathcal A_W^{-1}
\bigl(jx+N_f(\alpha,\beta,jx)\bigr)
\]
converges in \(X\). Hence this mapping is compact on bounded subsets of
\(\mathbb R^2\times X\).

Consequently,
\[
\Psi_W(\alpha,\beta,x)
=
x-\mathcal A_W^{-1}
\bigl(jx+N_f(\alpha,\beta,jx)\bigr)
\]
is a compact perturbation of the identity on \(X\).
\end{proof}

\begin{proposition}\label{prop:weyl-equivariance}
Assume that \(\mathcal A_W\) and \(N_f\) are \(G_W\)-equivariant. Then
\(\mathcal A_W^{-1}\) is also \(G_W\)-equivariant, and consequently
\(\Psi_W\) is \(G_W\)-equivariant; that is,
\[
\Psi_W(\alpha,\beta,hx)
=
h\Psi_W(\alpha,\beta,x),
\qquad
h\in G_W.
\]
\end{proposition}

\begin{proof}
Let \(h\in G_W\). Since \(\mathcal A_W\) is an equivariant isomorphism, its
inverse satisfies
\[
\mathcal A_W^{-1}(hy)
=
h\mathcal A_W^{-1}(y),
\qquad
y\in Y.
\]
Using the equivariance of \(N_f\), we obtain
\[
\begin{aligned}
\Psi_W(\alpha,\beta,hx)
&=
hx-\mathcal A_W^{-1}
\bigl(j(hx)+N_f(\alpha,\beta,j(hx))\bigr)\\
&=
hx-\mathcal A_W^{-1}
\bigl(h(jx+N_f(\alpha,\beta,jx))\bigr)\\
&=
h\left[
x-\mathcal A_W^{-1}
\bigl(jx+N_f(\alpha,\beta,jx)\bigr)
\right]\\
&=
h\Psi_W(\alpha,\beta,x).
\end{aligned}
\]
Therefore, \(\Psi_W\) is \(G_W\)-equivariant.
\end{proof}

Consider the periodic Riesz equation
\begin{equation}\label{eq:riesz-original-fixed-point}
(-\partial_t^2)^{q/2}x=N_g(\lambda,jx).
\end{equation}
Define
\[
\mathcal A_R:=(-\partial_t^2)^{q/2}+\Id.
\]

\begin{proposition}\label{prop:riesz-operator-properties}
The operator \(\mathcal A_R\) is densely defined, self-adjoint, and strictly
positive. Moreover,
\[
\mathcal A_R:
H_{\mathrm{per}}^q(S^1;V)
\longrightarrow
L^2(S^1;V)
\]
is bounded.
\end{proposition}

\begin{proof}
Let
\[
x(t)=\sum_{m\in\bz}x_m e^{imt}.
\]
Then
\[
\mathcal A_Rx(t)
=
\sum_{m\in\bz}
\bigl(1+|m|^q\bigr)x_m e^{imt}.
\]
Hence, by Parseval's identity,
\[
\begin{aligned}
\|\mathcal A_Rx\|_Y^2
&=
\sum_{m\in\bz}
\bigl(1+|m|^q\bigr)^2|x_m|^2\\
&\leq
4\sum_{m\in\bz}
(1+|m|^2)^q|x_m|^2\\
&=
4\|x\|_{H_{\mathrm{per}}^q}^2.
\end{aligned}
\]
Therefore,
\[
\mathcal A_R:
H_{\mathrm{per}}^q(S^1;V)\longrightarrow Y
\]
is bounded.

Since the Fourier multipliers \(1+|m|^q\) are real, for
\(x,y\in D(\mathcal A_R)\) one has
\[
\begin{aligned}
\langle \mathcal A_Rx,y\rangle_Y
&=
\sum_{m\in\bz}
\bigl(1+|m|^q\bigr)
\langle x_m,y_m\rangle_{V_{\mathbb C}}\\
&=
\langle x,\mathcal A_Ry\rangle_Y.
\end{aligned}
\]
Thus, \(\mathcal A_R\) is symmetric. Its Fourier representation shows that
the domain of its adjoint is again \(H_{\mathrm{per}}^q(S^1;V)\), and hence
\(\mathcal A_R\) is self-adjoint.

Finally,
\[
\langle \mathcal A_Rx,x\rangle_Y
=\sum_{m\in\bz}
\bigl(1+|m|^q\bigr)|x_m|^2
\geq
\|x\|_Y^2.
\]
Therefore, \(\mathcal A_R\) is strictly positive.
\end{proof}

Since
\[
1+|m|^q>0,\quad m\in\bz,
\]
the operator \(\mathcal A_R\) is an isomorphism from \(X\) onto \(Y\). Its inverse is
\[
\mathcal A_R^{-1}y(t)=\sum_{m\in\bz}
\frac{y_m}{1+|m|^q}e^{imt}.
\]
Similarly, equation~\eqref{eq:riesz-original-fixed-point} is equivalent to
\[
x=\mathcal A_R^{-1}\bigl(jx+N_g(\lambda,jx)\bigr).
\]
Define
\[
\Psi_R:\mathbb R\times X\longrightarrow X
\]
by
\[
\Psi_R(\lambda,x):=x-\mathcal A_R^{-1}
\bigl(jx+N_g(\lambda,jx)\bigr).
\]
Then
\[
\Psi_R(\lambda,x)=0
\]
if and only if \(x\) is a periodic solution of the Riesz equation.

By an analogous argument, the operator \(\Psi_R\) is a compact perturbation of the identity and \(G_R\)-equivariant.

Since
\[
f(\alpha,\beta,0)=0, \quad g(\lambda,0)=0,
\]
we have
\[
\Psi_W(\alpha,\beta,0)=0, \quad \Psi_R(\lambda,0)=0.
\]
Thus, the sets
\[
\mathcal T_W:=
\left\{(\alpha,\beta,0):(\alpha,\beta)\in\mathbb R^2\right\}
\]
and
\[
\mathcal T_R:=\left\{(\lambda,0):
\lambda\in\mathbb R\right\}
\]
are the trivial solution branches of the Weyl and Riesz problems, respectively.

A point \(
(\alpha^*,\beta^*,0)\in\mathcal T_W\)
is called a bifurcation point of the Weyl problem if every neighborhood of \((\alpha^*,\beta^*,0)\) contains a solution \((\alpha,\beta,x)\)
of \(\Psi_W(\alpha,\beta,x)=0\) with \(x\neq0\).

Similarly, a point \((\lambda^*,0)\in\mathcal T_R\)
is called a bifurcation point of the Riesz problem if every neighborhood of \((\lambda^*,0)\) contains a solution
\((\lambda,x)
\)
of \(\Psi_R(\lambda,x)=0\) with \(x\neq0\).

%----------------------------------------------
\section{Linearization and Spectral Decomposition}
%----------------------------------------------

Let
\[
B_W(\alpha,\beta):=D_xf(\alpha,\beta,0)
\]
and
\[
B_R(\lambda):=D_xg(\lambda,0).
\]
The linearizations of the fixed-point maps with respect to \(x\) at the trivial
solutions are
\[
D_x\Psi_W(\alpha,\beta,0)
=\Id-\mathcal A_W^{-1}
\bigl(\Id+B_W(\alpha,\beta)\bigr),
\]
and
\begin{equation}\label{eq:linearization-riesz}
D_x\Psi_R(\lambda,0)=
\Id-\mathcal A_R^{-1}(\Id+B_R(\lambda))
\end{equation}

A parameter value can be critical only if the corresponding linearized
fixed-point operator is not invertible. Thus, the Weyl critical set is
\[
\Lambda_W:=\left\{
(\alpha,\beta)\in\mathbb R^2:
D_x\Psi_W(\alpha,\beta,0)
\text{ is not invertible}
\right\},
\]
while the Riesz critical set is
\[
\Lambda_R:=
\left\{
\lambda\in\mathbb R:
D_x\Psi_R(\lambda,0)
\text{ is not invertible}
\right\}.
\]

These fixed-point formulations provide the analytic setting required for the
twisted equivariant degree in the Weyl problem and the ordinary equivariant
degree in the Riesz problem.

Since \(f\) and \(g\) are \(\mathcal G\)-equivariant, the linear operators \(B_W(\alpha,\beta)\) and \(B_R(\lambda)\) commute with the \(\mathcal G\)-action. As the additional \(\bz_2\)-action is given by sign reversal, these operators also commute with the action of \(\mathcal G\times\bz_2\). Hence, each \(\mathcal G\times\bz_2\)-isotypical component of \(V\) is invariant under \(B_W(\alpha,\beta)\) and \(B_R(\lambda)\).

Let
\[
V=V_0\oplus\cdots\oplus V_r
\]
be the \(\mathcal G\times\bz_2\)-isotypical decomposition of \(V\). For each
\(j=0,\ldots,r\), assume that \(V_j\) is modeled on an irreducible
\(\mathcal G\times\bz_2\)-representation \(\mathcal V_j\); that is,
\[
V_j\simeq \mathcal V_j^{\oplus n_j}
\]
for some multiplicity \(n_j\geq1\).

To obtain the corresponding decomposition of the periodic function space, let
\[
X_0:=\left\{
x:S^1\to V:
x(t)\equiv v,\ v\in V
\right\}
\]
denote the space of \(V\)-valued constant functions. For \(m\geq1\), define
\[
X_m:=\left\{
u\cos(mt)+v\sin(mt):
u,v\in V
\right\}.
\]
Then
\[
X=H_{\mathrm{per}}^q(S^1;V)
=
\overline{\bigoplus_{m=0}^{\infty}X_m}.
\]

Let
\[
E_0:=\mathbb R,
\qquad
E_m:=
\operatorname{span}_{\mathbb R}
\{\cos(mt),\sin(mt)\},
\qquad m\geq1.
\]
For \(j=0,\ldots,r\), define
\[
V_{j,0}:=
\left\{
x:S^1\to V_j:
x(t)\equiv v_j,\ v_j\in V_j
\right\},
\]
and, for \(m\geq1\),
\[
V_{j,m}:=
\left\{
u_j\cos(mt)+v_j\sin(mt):
u_j,v_j\in V_j
\right\}.
\]
Equivalently,
\[
V_{j,m}=V_j\otimes E_m,
\qquad
j=0,\ldots,r,\quad m\geq0.
\]
If
\[
V_j\simeq\mathcal V_j^{\oplus n_j},
\]
then the corresponding irreducible spatio-temporal representation is
\[
\mathcal V_{j,m}:=
\mathcal V_j\otimes E_m,
\]
and therefore
\[
V_{j,m}\simeq
\mathcal V_{j,m}^{\oplus n_j}.
\]
Thus,
\[
X_m=\bigoplus_{j=0}^{r}V_{j,m},
\qquad m\geq0,
\]
and consequently
\[
X=
\overline{
\bigoplus_{j=0}^{r}
\bigoplus_{m=0}^{\infty}
V_{j,m}
}.
\]

Each block \(V_{j,m}\) is invariant under the linearized Weyl and
Riesz operators. Therefore, the spectral analysis of
\(D_x\Psi_W(\alpha,\beta,0)\) and \(D_x\Psi_R(\lambda,0)\) reduces to the
study of their restrictions to the finite-dimensional blocks
\(V_{j,m}\).

%----------------------------------------------------
\subsection{Critical Points of the Weyl Equation}
%----------------------------------------------------

Define
\[
L_W(\alpha,\beta):=D_{W,+}^{q}-B_W(\alpha,\beta).
\]
A parameter pair \((\alpha,\beta)\) is critical if and only if \(L_W(\alpha,\beta)\) has a nontrivial kernel.

On the complex temporal Fourier mode \(e^{imt}\), one has
\[
D_{W,+}^{q}e^{imt}=(im)^q e^{imt},
\]
where
\[
(im)^q=|m|^q
\left[
\cos\left(\frac{\pi q}{2}\right)
+i\,\operatorname{sgn}(m)
\sin\left(\frac{\pi q}{2}\right)
\right].
\]

Let
\[
\eta_{j,k}(\alpha,\beta)
=
a_{j,k}(\alpha,\beta)
+i\,b_{j,k}(\alpha,\beta),
\qquad
k=1,\ldots,s_j,
\]
denote the eigenvalues of the complexification of
\(B_{W,j}(\alpha,\beta)\), where \(k\) indexes the distinct eigenvalue branches on the spatial isotypical component \(V_j\).

Then the eigenvalue of \(L_W(\alpha,\beta)\) on the corresponding spatio-temporal mode is
\[
\Delta_{j,k,m}^{W}(\alpha,\beta)
=(im)^q-\eta_{j,k}(\alpha,\beta).
\]

For \(m\geq1\), this becomes
\[
\Delta_{j,k,m}^{W}(\alpha,\beta)
=m^q\cos\left(\frac{\pi q}{2}\right)
-a_{j,k}(\alpha,\beta)+i\left[
m^q\sin\left(\frac{\pi q}{2}\right)
-b_{j,k}(\alpha,\beta)
\right].
\]

Hence the block \(V_{j,m}\) is critical precisely when
\begin{equation}\label{eq:critical}
\begin{aligned}
&a_{j,k}(\alpha,\beta)=
m^q\cos\left(\frac{\pi q}{2}\right)\\
&b_{j,k}(\alpha,\beta)=
m^q\sin\left(\frac{\pi q}{2}\right)
\end{aligned}
\end{equation}
for some \(k\).

Thus, for a nonzero temporal mode, criticality requires the simultaneous satisfaction of two real equations.

For the zero temporal mode, the Weyl multiplier is zero. Hence
\[
\Delta_{j,k,0}^{W}(\alpha,\beta)
=-\eta_{j,k}(\alpha,\beta),
\]
and the zero mode is critical when
\[
\eta_{j,k}(\alpha,\beta)=0.
\]

The eigenvalue corresponding to the \(j,m\)-block of the linearized Weyl fixed-point map is given by
\[
\mu_{j,k,m}^{W}(\alpha,\beta)
=1-\frac{1+\eta_{j,k}(\alpha,\beta)}
{1+(im)^q}=
\frac{(im)^q-\eta_{j,k}(\alpha,\beta)}
{1+(im)^q}.
\]
Since
\[
1+(im)^q\neq0,
\]
the fixed-point eigenvalue vanishes precisely when
\[
\Delta_{j,k,m}^{W}(\alpha,\beta)=0.
\]

%----------------------------------------------------
\subsection{Critical Points of the Riesz Equation}
%----------------------------------------------------
Define
\[
L_R(\lambda):=
(-\partial_t^2)^{q/2}-B_R(\lambda).
\]

Assume that \(B_{R,j}(\lambda)\) is self-adjoint and let
\[
\nu_{j,k}(\lambda)\in\mathbb R
\]
be one of its eigenvalues. Since
\[
(-\partial_t^2)^{q/2}e^{imt}
=|m|^q e^{imt},
\]
the eigenvalue of \(L_R(\lambda)\) on the corresponding \((j,m)\)-block is
\[
\Delta_{j,k,m}^{R}(\lambda)=
|m|^q-\nu_{j,k}(\lambda).
\]

For \(m\geq1\), the block \(V_{j,m}\) is critical precisely when
\[
\nu_{j,k}(\lambda)=m^q \quad \text{for some } k.
\]
Unlike the Weyl case, this is one real equation in the single real parameter \(\lambda\). Therefore, isolated critical values occur generically in a one-parameter Riesz family.

For the zero temporal mode,
\[
\Delta_{j,k,0}^{R}(\lambda)
=-\nu_{j,k}(\lambda),
\]
so the constant mode is critical when
\[
\nu_{j,k}(\lambda)=0.
\]
The eigenvalues corresponding to the \((j,m)\)-block of the linearized Riesz fixed-point map~\eqref{eq:linearization-riesz} are given by
\[
\mu_{j,k,m}^{R}(\lambda)
=1-\frac{1+\nu_{j,k}(\lambda)}
{1+|m|^q}=
\frac{|m|^q-\nu_{j,k}(\lambda)}{1+|m|^q}.
\]
Since
\[
1+|m|^q>0,
\]
one has
\[
\mu_{j,k,m}^{R}(\lambda)=0
\quad\Longleftrightarrow\quad
\Delta_{j,k,m}^{R}(\lambda)=0.
\]

%----------------------------------------------------
\subsection{Critical parameter sets}
%----------------------------------------------------

The critical set of the Weyl problem is
\[
\Lambda_W=\bigcup_{j=0}^{r}\bigcup_{k=1}^{s_j}
\bigcup_{m=0}^{\infty}\left\{
(\alpha,\beta)\in\mathbb R^2:
\Delta_{j,k,m}^{W}(\alpha,\beta)=0
\right\}.
\]
For \(m\geq1\), each condition \(\Delta_{j,k,m}^{W}=0\) consists of two real equations. Under the nondegeneracy condition
\[
\det D_{(\alpha,\beta)}
\begin{pmatrix}
\operatorname{Re}\Delta_{j,k,m}^{W}\\
\operatorname{Im}\Delta_{j,k,m}^{W}
\end{pmatrix}
(\alpha^*,\beta^*)
\neq0,
\]
the critical parameter pair \((\alpha^*,\beta^*)\) is isolated.

The critical set of the Riesz problem is
\[
\Lambda_R=\bigcup_{j=0}^{r}\bigcup_{k=1}^{s_j}
\bigcup_{m=0}^{\infty}
\left\{\lambda\in\mathbb R:
\Delta_{j,k,m}^{R}(\lambda)=0
\right\}.
\]

The spectral formulas obtained above show the essential difference between the two problems. The one-sided Weyl equation is governed by complex eigenvalue maps whose zeros generically require two parameters, whereas the Riesz equation is governed by real eigenvalue crossings that can occur generically in a one-parameter family.

%----------------------------------------------------
\section{Winding Numbers for the Weyl Problem}
%----------------------------------------------------
On each spatial isotypical component \(V_j\) and each nonzero temporal mode \(m\geq1\), the linearized one-sided Weyl problem is governed by the complex-valued characteristic function
\[
\Delta_{j,k,m}^{W}(\alpha,\beta)=
u_{j,k,m}(\alpha,\beta)
+i\,v_{j,k,m}(\alpha,\beta),
\]
where
\[
u_{j,k,m}(\alpha,\beta)=
m^q\cos\left(\frac{\pi q}{2}\right)-a_{j,k}(\alpha,\beta)
\]
and
\[
v_{j,k,m}(\alpha,\beta)=
m^q\sin\left(\frac{\pi q}{2}\right)-b_{j,k}(\alpha,\beta).
\]

Thus, the \((j,k,m)\)-block is critical precisely when
\[
u_{j,k,m}(\alpha,\beta)=0
\]
and
\[
v_{j,k,m}(\alpha,\beta)=0.
\]
Equivalently,
\[
\Delta_{j,k,m}^{W}(\alpha,\beta)=0.
\]

Let \(\mathcal D\subset\mathbb R^2\) be a bounded open set such that
\[
\Delta_{j,k,m}^{W}(\alpha,\beta)\neq0
\qquad
\text{for every }(\alpha,\beta)\in\partial \mathcal D.
\]
The winding number of the image curve
\(\Delta_{j,k,m}^{W}(\partial \mathcal D)\) around the origin is defined by
\[
\operatorname{wind}
\bigl(\Delta_{j,k,m}^{W},
\partial {\mathcal D}, 0\bigr):=
\frac{1}{2\pi i}\int_{\partial \mathcal D}
\frac{d\Delta_{j,k,m}^{W}}
{\Delta_{j,k,m}^{W}}.
\]

Identifying \(\mathbb C\) with \(\mathbb R^2\), the winding number is equivalently the Brouwer degree of the real map
\[
(\alpha,\beta)
\longmapsto
\bigl(
u_{j,k,m}(\alpha,\beta),
v_{j,k,m}(\alpha,\beta)
\bigr).
\]
Hence,
\[
\operatorname{wind}
\bigl(
\Delta_{j,k,m}^{W},
\partial \mathcal D,0\bigr)
=\deg_{\mathrm B}\left(
\bigl(u_{j,k,m},v_{j,k,m}\bigr),\mathcal D,0
\right).
\]

Suppose that \((\alpha^*,\beta^*)\in \mathcal D\) is the unique zero of \(\Delta_{j,k,m}^{W}\) in \(\mathcal D\), and assume that the zero is nondegenerate:
\[
\det
D_{(\alpha,\beta)}
\begin{pmatrix}
u_{j,k,m}\\
v_{j,k,m}
\end{pmatrix}
(\alpha^*,\beta^*)
\neq 0.
\]
Then, for a sufficiently small disk
\(\mathcal D_\varepsilon(\alpha^*,\beta^*)\), the local winding number is
\[
\rho_{j,k,m}^{W}:=
\operatorname{wind}
\bigl(\Delta_{j,k,m}^{W},
\partial \mathcal D_\varepsilon(\alpha^*,\beta^*),
0\bigr)
=\operatorname{sgn}
\det 
D_{(\alpha,\beta)}
\begin{pmatrix}
u_{j,k,m}\\
v_{j,k,m}
\end{pmatrix}
(\alpha^*,\beta^*).
\]

Since the temporal multiplier \((im)^q\) is independent of the parameters,
\[
\begin{pmatrix}
u_{j,k,m}\\
v_{j,k,m}
\end{pmatrix}
=
\begin{pmatrix}
m^q\cos\left(\dfrac{\pi q}{2}\right)\\[1mm]
m^q\sin\left(\dfrac{\pi q}{2}\right)
\end{pmatrix}
-
\begin{pmatrix}
a_{j,k}\\
b_{j,k}
\end{pmatrix}.
\]
Consequently,
\[
D_{(\alpha,\beta)}
\begin{pmatrix}
u_{j,k,m}\\
v_{j,k,m}
\end{pmatrix}
=-
D_{(\alpha,\beta)}
\begin{pmatrix}
a_{j,k}\\
b_{j,k}
\end{pmatrix}.
\]
Because the parameter space has dimension two,
\[
\det(-I_2)=1.
\]
Therefore, multiplication by \(-I_2\) preserves orientation, and
\[
\rho_{j,k,m}^{W}
=\operatorname{sgn}
\det
D_{(\alpha,\beta)}
\begin{pmatrix}
a_{j,k}\\
b_{j,k}
\end{pmatrix}
(\alpha^*,\beta^*).
\]

For real-valued periodic functions, the temporal modes \(m\) and \(-m\) occur as a complex-conjugate pair and together generate the real temporal
space
\[
\operatorname{span}_{\mathbb R}
\left\{
\sin(mt),\cos(mt)
\right\}.
\]
Therefore, it is sufficient to index the nonconstant real Fourier blocks by \(m\geq1\). The winding number \(\rho_{j,k,m}^{W}\) is attached to the
corresponding real \(\mathcal G\times\bz_2\times S^1\)-representation \(V_{j,m}\).

If several isolated critical blocks lie inside \(\mathcal D\), the winding number of the determinant of the corresponding complex linearized block is obtained by adding the local winding numbers, counted with their algebraic
multiplicities.
%----------------------------------------------------
\section{Local Bifurcation for the One-Sided Weyl Problem}
\label{sec:weyl-local-bifurcation}
%----------------------------------------------------

In this section, we compute the local \(G_W\)-equivariant bifurcation invariant associated with an isolated critical parameter pair of the one-sided periodic Weyl equation. The calculation follows four steps. First, the nonlinear fixed-point field is completed by an auxiliary scalar component. Second, the resulting field is reduced to its linearization along the trivial branch. Third, the linearized
field is decomposed into its stationary and nonstationary Fourier components. Finally, the nonstationary contribution is expressed through the winding numbers of the critical complex Fourier blocks.

The derivative of \(\Psi_W\) with respect to \(x\) along the trivial branch is
\[
T_W(\alpha,\beta):=
D_x\Psi_W(\alpha,\beta,0)
=\Id-\mathcal A_W^{-1}
\bigl(\Id+B_W(\alpha,\beta)\bigr).
\]

A parameter pair \(p=(\alpha,\beta)\) is called regular if \(T_W(p)\) is an isomorphism. Let
\[
p^*=(\alpha^*,\beta^*)
\]
be an isolated critical parameter pair. Hence, there exists
\(\varepsilon>0\) such that
\[
T_W(p)
\text{ is an isomorphism whenever }
0<|p-p^*|\leq\varepsilon.
\]

We additionally assume that the stationary block is nonsingular:
\[
T_{W,0}(p^*):=
T_W(p^*)\big|_{X_0}
\quad\text{is an isomorphism}.
\]
This assumption excludes bifurcation caused exclusively by the constant temporal mode.

%----------------------------------------------------
\subsection{Definition of the local invariant}
%----------------------------------------------------

Choose \(\delta>0\) sufficiently small and define
\[
\Omega_W:=
\left\{
(p,x)\in\mathbb R^2\times X:
|p-p^*|<\varepsilon,\ 
\|x\|_X<\delta
\right\}.
\]

Let
\[
\vartheta_W:\overline{\Omega_W}\longrightarrow\mathbb R
\]
be the \(G_W\)-invariant auxiliary function
\[
\vartheta_W(p,x):=|p-p^*|
\bigl(\|x\|_X-\delta\bigr)
+\|x\|_X-\frac{\delta}{2}.
\]
The action of \(G_W\) on the parameter variables is trivial, and therefore \(\vartheta_W\) is \(G_W\)-invariant.

For \(x=0\) and \(|p-p^*|=\varepsilon\), after a harmless positive rescaling of the parameter variables if necessary, one has
\[
\vartheta_W(p,0)<0.
\]
On the lateral boundary \(\|x\|_X=\delta\), one has
\[
\vartheta_W(p,x)>0.
\]

Define the complemented field
\[
\mathfrak F_W(p,x):=
\bigl(\vartheta_W(p,x),
\Psi_W(p,x)\bigr),
\qquad
(p,x)\in\overline{\Omega_W}.
\]
For \(\varepsilon\) and \(\delta\) sufficiently small,
\(\mathfrak F_W\) has no zero on \(\partial\Omega_W\). Hence, its twisted \(G_W\)-equivariant degree is well defined.

\begin{definition}\label{def:weyl-local-invariant}
The local Weyl bifurcation invariant at \((p^*,0)\) is
\[
\omega_W[p^*,0]:=
G_W\text{-}\deg
\bigl(
\mathfrak F_W,\Omega_W
\bigr)
\in A_1^t(G_W),
\]
where \(A_1^t(G_W)\) denotes the twisted equivariant-degree group associated with orbit types having one-dimensional Weyl group.
\end{definition}

%----------------------------------------------------
\subsection{Reduction to the linearized field}
%----------------------------------------------------

The first step is to replace the nonlinear field by its derivative along the trivial branch.

\begin{proposition}\label{prop:weyl-linear-reduction}
For \(\varepsilon>0\) and \(\delta>0\) sufficiently small,
\(\mathfrak F_W\) is admissibly \(G_W\)-homotopic in \(\Omega_W\) to
\[
\mathfrak L_W(p,x):=
\bigl(
\vartheta_W(p,x),
T_W(p)x
\bigr).
\]
Consequently,
\[
\omega_W[p^*,0]=
G_W\text{-}\deg
\bigl(\mathfrak L_W,\Omega_W\bigr).
\]
\end{proposition}

\begin{proof}
Since \(\Psi_W\) is continuously Fr\'echet differentiable with respect to
\(x\), one has
\[
\Psi_W(p,x)
=
T_W(p)x+R_W(p,x),
\]
where
\[
\frac{\|R_W(p,x)\|_X}{\|x\|_X}
\longrightarrow 0
\]
uniformly for \(p\in\overline{B_\varepsilon(p^*)}\) as
\(\|x\|_X\to0\).

For \(s\in[0,1]\), define
\[
H_s(p,x)
=\bigl(\vartheta_W(p,x),
T_W(p)x+sR_W(p,x)\bigr).
\]
On the lateral boundary \(\|x\|_X=\delta\), the first component is positive,
so \(H_s(p,x)\neq0\).

On the parameter boundary \(|p-p^*|=\varepsilon\), the operators \(T_W(p)\) are invertible. Compactness of the
boundary circle gives a constant \(c>0\) such that
\[
\|T_W(p)x\|_X
\geq
c\|x\|_X.
\]
One may assume that
\[
\|R_W(p,x)\|_X
\leq
\frac{c}{2}\|x\|_X.
\]
Therefore, for \(x\neq0\),
\[
\begin{aligned}
\|T_W(p)x+sR_W(p,x)\|_X
&\geq
\|T_W(p)x\|_X-s\|R_W(p,x)\|_X\\
&\geq
\frac{c}{2}\|x\|_X>0.
\end{aligned}
\]
When \(x=0\), the auxiliary component is nonzero on the parameter boundary.
Thus, \(H_s\) is an admissible \(G_W\)-equivariant homotopy, and the
homotopy property of the twisted degree proves the result.
\end{proof}

%----------------------------------------------------
\subsection{Stationary and dynamic Fourier components}
%----------------------------------------------------

Recall the Fourier--isotypical decomposition
\[
X=\overline{
\bigoplus_{j=0}^{r}
\bigoplus_{m=0}^{\infty}V_{j,m}
},
\qquad
V_{j,m}=V_j\otimes E_m.
\]

Set
\[
X_+:=\overline{\bigoplus_{j=0}^{r}
\bigoplus_{m=1}^{\infty}
V_{j,m}}.
\]
Thus,
\[
X=X_0\oplus X_+.
\]

Since \(T_W(p)\) is \(G_W\)-equivariant, it preserves every
\(V_{j,m}\). We write
\[
T_{W,j,m}(p):=
T_W(p)\big|_{V_{j,m}}.
\]
Define
\[
\mathfrak L_W^+(p,x_+)=
\bigl(\vartheta_W(p,x_+),T_W^+(p)x_+\bigr),
\qquad
T_W^+(p):=T_W(p)\big|_{X_+},
\]
and
\[
\Omega_W^+=
\left\{
(p,x_+)\in\mathbb R^2\times X_+:
|p-p^*|<\varepsilon,\ 
\|x_+\|_X<\delta
\right\}.
\]
By an admissible equivariant homotopy and the multiplication property of the equivariant degree, we obtain
\begin{equation}\label{eq:weyl-stationary-dynamic-product}
\omega_W[p^*,0]=
G_0\text{-}\deg
\bigl(T_{W,0}(p^*),B_\zeta(X_0)
\bigr)
\cdot G_W\text{-}\deg
\bigl(\mathfrak L_W^+,\Omega_W^+ \bigr),
\end{equation}
where \(
G_0:=\mathcal G\times\bz_2\), 
\(B_\zeta(X_0):=\left\{
x\in X_0:|x|<\zeta
\right\}\)
denotes the open \(\zeta\)-ball centered at the origin in \(X_0\) and \(\mathfrak L_W^+\) denotes the complemented field restricted to
\(\mathbb R^2\times X_+\).

\begin{definition}
Let \(\mathcal V_{j,m}\) be an irreducible \(G_W\)-representation. The \(G_W\)-equivariant basic degree corresponding to \(\mathcal V_{j,m}\) is
defined by
\[
\deg_{\mathcal V_{j,m}}^{\,G_W}:=
\eqdeg{G_W}\left(-\Id,
B_\zeta(\mathcal V_{j,m})
\right),
\]
where \(B_\zeta(\mathcal V_{j,m})\) denotes the open \(\zeta\)-ball centered at the origin in \(\mathcal V_{j,m}\), see~\cite{AED}.
\end{definition}

Let
\[
\sigma_-\bigl(T_{W,0}(p^*)\bigr)
\]
denote the set of negative real eigenvalues of the stationary block. For
\(\mu\in\sigma_-(T_{W,0}(p^*))\), let \(\mathfrak m_j(\mu)\) be the multiplicity of
\(\mathcal V_{j,0}\) in the corresponding generalized eigenspace. Since only negative real eigenvalues
contribute nontrivially to the degree of a real linear isomorphism, the stationary degree is given by
\begin{equation}\label{eq:weyl-stationary-degree}
G_0\text{-}\deg
\bigl(T_{W,0}(p^*),B_\zeta(X_0)\bigr)
=
\prod_{\mu\in\sigma_-(T_{W,0}(p^*))}
\prod_{j=0}^{r}
\left(
\deg_{\mathcal V_{j,0}}^{\,G_0}
\right)^{\mathfrak m_j(\mu)}
\in A(G_0).
\end{equation}
Here, \(A(G_0)\) denotes the Burnside ring generated by conjugacy classes \((H)\) of subgroups of \(G_0\) whose Weyl groups
\[
W_{G_0}(H)=N_{G_0}(H)/H
\]
have dimension zero.

The positive real spectral blocks and the real invariant blocks associated with nonreal complex-conjugate eigenvalue pairs belong to the identity component of the corresponding space of linear isomorphisms and therefore contribute the multiplicative identity to the degree. Hence only the negative real eigenvalues appear explicitly in the above product. If the stationary block has no negative real eigenvalues, the product is understood as the multiplicative identity.

%----------------------------------------------------
\subsection{Dynamic winding contributions}
%----------------------------------------------------

Recall
\[
\operatorname{wind}
\bigl(\mu_{j,k,m}^{W},\partial \mathcal D,0
\bigr)
=\operatorname{wind}
\bigl(
\Delta_{j,k,m}^{W},\partial \mathcal D,0
\bigr)=
\rho_{j,k,m}^{W}.
\]

Let
\[
\mathcal C_W(p^*):=
\left\{(j,k,m):
m\geq1,\Delta_{j,k,m}^{W}(p^*)=0
\right\}.
\]
The set \(\mathcal C_W(p^*)\) is finite, since
\[
|(im)^q|=m^q\longrightarrow\infty
\qquad\text{as }m\to\infty.
\]
Hence, only finitely many temporal modes can satisfy the criticality condition at the fixed parameter value \(p^*\).
\begin{lemma}\label{lem:weyl-dynamic-block-degree}
Let \((j,k,m)\in\mathcal C_W(p^*)\) with \(m\geq1\). Assume that
\[
\Delta_{j,k,m}^{W}(p)\neq0,
\qquad
p\in\partial\mathcal D,
\]
and that all remaining characteristic factors of the \((j,m)\)-block
are nonzero on \(\overline{\mathcal D}\). Then
\[
G_W\text{-}\deg
\bigl(\mathfrak L_{W,j,m},
\Omega_{W,j,m}\bigr)=
d_{j,k,m}\rho_{j,k,m}^{W}\deg_{\mathcal V_{j,m}}^{\,G_W},
\]
where \(d_{j,k,m}\) is multiplicity of the irreducible critical representation/eigenvalue branch in the corresponding isotypical block and 
\[
\rho_{j,k,m}^{W}
=
\operatorname{wind}
\left(
\Delta_{j,k,m}^{W},
\partial\mathcal D,0
\right).
\]
\end{lemma}
\begin{proof}
Fix \((j,k,m)\in\mathcal C_W(p^*)\). By the standard
finite-dimensional reduction for compact equivariant fields, the \((j,m)\)-block decomposes into a finite-dimensional critical part and an invariant complementary subspace on which the family remains uniformly invertible. The complementary part can therefore be deformed, through \(G_W\)-equivariant isomorphisms, to the identity and hence
contributes only the multiplicative identity to the equivariant degree.

Since \(m\geq1\), the real temporal Fourier space
\[
\operatorname{span}_{\mathbb R}
\{\cos(mt),\sin(mt)\}
\]
carries the natural complex structure induced by the \(S^1\)-action. Consequently, after the finite-dimensional reduction, the critical part of the \((j,m)\)-block is represented by a continuous complex
matrix family
\[
M_{j,m}^{W}(p),
\qquad
p\in\overline{\mathcal D}.
\]
Let \(d\) denote the complex dimension of this reduced block. By the
hypotheses,
\[
\det_{\mathbb C}M_{j,m}^{W}(p)\neq0,
\qquad
p\in\partial\mathcal D,
\]
and therefore
\[
M_{j,m}^{W}|_{\partial\mathcal D}:
\partial\mathcal D\simeq S^1
\longrightarrow
GL(d,\mathbb C)
\]
defines a loop of invertible complex matrices.

By the isotypical crossing-number formula for the twisted equivariant degree \cite{AED}, the contribution of the \((j,m)\)-block is determined by the integer represented by this complex matrix loop.
Since
\[
\pi_1\bigl(GL(d,\mathbb C)\bigr)\simeq\bz
\]
and the determinant induces an isomorphism
\[
\det_*:
\pi_1\bigl(GL(d,\mathbb C)\bigr)
\longrightarrow
\pi_1(\mathbb C^\times)
\simeq\bz,
\]
this integer is precisely the winding number of the determinant.
Hence
\begin{equation}\label{eq:weyl-determinantal-reduction}
G_W\text{-}\deg
\bigl(
\mathfrak L_{W,j,m},
\Omega_{W,j,m}\bigr)=
\operatorname{wind}
\left(\det_{\mathbb C}M_{j,m}^{W},
\partial\mathcal D,0
\right)\deg_{\mathcal V_{j,m}}^{\,G_W}.
\end{equation}
Since \(M_{j,m}^{W}(p)\) is the reduced matrix associated with the irreducible component \(\mathcal V_{j,m}\), the critical factor \(\Delta_{j,k,m}^{W}\) appears with multiplicity \(d_{j,k,m}\). As all other characteristic factors remain nonzero on \(\overline{\mathcal D}\), we may write
\[
\det_{\mathbb C}M_{j,m}^{W}(p)
=h_{j,k,m}(p)
\bigl(\Delta_{j,k,m}^{W}(p)\bigr)^{d_{j,k,m}},
\]
where \(h_{j,k,m}\) is continuous and nonvanishing on \(\overline{\mathcal D}\).

Since \(h_{j,k,m}\) extends as a nonvanishing map over
\(\overline{\mathcal D}\), its restriction to
\(\partial\mathcal D\) is null-homotopic in \(\mathbb C^\times\).
Therefore,
\[
\operatorname{wind}
\left(h_{j,k,m},
\partial\mathcal D,0\right)=0.
\]

Using the additivity of the winding number under multiplication, we
obtain
\[
\begin{aligned}
\operatorname{wind}
\left(
\det_{\mathbb C}M_{j,m}^{W},
\partial\mathcal D,0
\right)
&=d_{j,k,m}
\operatorname{wind}
\left(\Delta_{j,k,m}^{W},
\partial\mathcal D,0\right)
\\
&=
d_{j,k,m}\rho_{j,k,m}^{W}.
\end{aligned}
\]
Substituting this identity into
\eqref{eq:weyl-determinantal-reduction} yields
\[
G_W\text{-}\deg
\bigl(
\mathfrak L_{W,j,m},
\Omega_{W,j,m}
\bigr)=d_{j,k,m}\rho_{j,k,m}^{W}
\deg_{\mathcal V_{j,m}}^{\,G_W},
\]
which completes the proof.
\end{proof}

By the splitting and additivity properties of the twisted degree,
\begin{equation}\label{eq:weyl-total-dynamic-degree}
G_W\text{-}\deg
\bigl(
\mathfrak L_W^+,
\Omega_W^+
\bigr)
=
\sum_{(j,k,m)\in\mathcal C_W(p^*)}
d_{j,k,m}\rho_{j,k,m}^{W}
\deg_{\mathcal V_{j,m}}^{\,G_W}.
\end{equation}

%----------------------------------------------------
\subsection{Formula for the local Weyl invariant}
%----------------------------------------------------

\begin{theorem}[Computation of the local Weyl invariant]
\label{thm:weyl-local-invariant-formula}
Let \(p^*=(\alpha^*,\beta^*)\) be an isolated dynamic critical parameter pair. Assume that the stationary block \(T_{W,0}(p^*)\) is an isomorphism.
Then
\begin{equation}\label{eq:weyl-local-invariant-formula}
\begin{aligned}
\omega_W[p^*,0]
&=
\prod_{\mu\in\sigma_-(T_{W,0}(p^*))}
\prod_{j=0}^{r}
\left(
\deg_{\mathcal V_{j,0}}^{\,G_0}
\right)^{\mathfrak m_j(\mu)}
\cdot
\left(
\sum_{(j,k,m)\in\mathcal C_W(p^*)}
d_{j,k,m}\rho_{j,k,m}^{W}
\deg_{\mathcal V_{j,m}}^{\,G_W}
\right).
\end{aligned}
\end{equation}
\end{theorem}

\begin{proof}
Proposition~\ref{prop:weyl-linear-reduction} reduces the nonlinear complemented field to its linearization. Formula
\eqref{eq:weyl-stationary-dynamic-product} separates the stationary and dynamic parts. The stationary factor is given by
\eqref{eq:weyl-stationary-degree}, while the dynamic contribution is given by \eqref{eq:weyl-total-dynamic-degree}. Their product yields \eqref{eq:weyl-local-invariant-formula}.
\end{proof}
\begin{remark}
Since
\[
G_W=G_0\times S^1,
\]
there is a natural inflation homomorphism
\[
\iota:A(G_0)\longrightarrow A(G_W),
\qquad
\iota\bigl([G_0/H]\bigr)
=
[G_W/(H\times S^1)],
\]
obtained by letting the \(S^1\)-factor act trivially. Hence, the stationary factor in \(A(G_0)\) is identified with its image in \(A(G_W)\), while the dynamic factor belongs to the first twisted Burnside group \(A_1^t(G_W)\). Their product is understood through the natural action
\[
A(G_W)\times A_1^t(G_W)
\longrightarrow
A_1^t(G_W),
\]
and therefore
\[
\omega_W[p^*,0]\in A_1^t(G_W).
\]
\end{remark}

\begin{remark}\label{rem:weyl-trivial-stationary-factor}
If \(T_{W,0}(p^*)\) is equivariantly homotopic through isomorphisms to the
identity, then
\[
G_0\text{-}\deg
\bigl(
T_{W,0}(p^*),B_\zeta(X_0)
\bigr)
=
\mathbf 1.
\]
In this case,
\[
\omega_W[p^*,0]
=
\sum_{(j,k,m)\in\mathcal C_W(p^*)}
d_{j,k,m}\rho_{j,k,m}^{W}
\deg_{\mathcal V_{j,m}}^{\,G_W}.
\]
\end{remark}

%----------------------------------------------------
\subsection{Local bifurcation theorem}
%----------------------------------------------------

\begin{theorem}[Local bifurcation for the Weyl problem]
\label{thm:weyl-local-bifurcation}
Assume the hypotheses of
Theorem~\ref{thm:weyl-local-invariant-formula}. If
\begin{equation}\label{eq:weyl-nonzero-local-invariant}
\omega_W[p^*,0]\neq0
\qquad\text{in }A_1^t(G_W),
\end{equation}
then
\[
(\alpha^*,\beta^*,0)
\]
is a bifurcation point of nontrivial \(2\pi\)-periodic solutions of the
one-sided Weyl equation. More precisely, every neighborhood of
\[
(\alpha^*,\beta^*,0)
\]
contains a solution
\[
(\alpha,\beta,x)
\]
such that
\[
x\neq0,
\qquad
\Psi_W(\alpha,\beta,x)=0.
\]

Moreover, let \((H)\) be a maximal orbit type occurring with nonzero
coefficient in
\[
\omega_W[p^*,0].
\]
Then there exist nontrivial bifurcating periodic solutions whose isotropy group contains a subgroup conjugate to \(H\). In particular, if \((H)\) is maximal among the admissible orbit types in the corresponding critical representation, then the bifurcating solutions detected by this coefficient have isotropy type \((H)\).
\end{theorem}
\begin{remark}
    Since the bifurcation invariant is generated by nonzero temporal modes \(m\geq1\), the detected bifurcating solutions are genuinely time-dependent
and hence nonconstant.
\end{remark}
\begin{proof}
Suppose, to the contrary, that
\[
(\alpha^*,\beta^*,0)
\]
is not a bifurcation point. Then, after reducing
\(\varepsilon\) and \(\delta\) if necessary, the only zeros of
\(\Psi_W\) in \(\overline{\Omega_W}\) belong to the trivial branch.

The auxiliary scalar component separates the trivial branch from the boundary of the isolating neighborhood. Hence the complemented field may be deformed, through an admissible \(G_W\)-equivariant homotopy, to a field having no zeros in \(\Omega_W\). By the existence and homotopy properties of the twisted equivariant degree,
\[
G_W\text{-}\deg
\bigl(
\mathfrak F_W,\Omega_W
\bigr)
=
0.
\]
By Definition~\ref{def:weyl-local-invariant}, it follows that
\[
\omega_W[p^*,0]=0,
\]
contradicting
\eqref{eq:weyl-nonzero-local-invariant}. Therefore, nontrivial solutions occur arbitrarily close to the critical point.

Now let \((H)\) be a maximal orbit type occurring with nonzero coefficient in \(\omega_W[p^*,0]\). By the existence property of the twisted equivariant degree, there exists a nontrivial bifurcating solution \(x\) whose isotropy group contains a subgroup conjugate to \(H\); equivalently,
\[
(H)\leq (G_x).
\]
If \((H)\) is maximal among the admissible orbit types in the corresponding
critical representation, then necessarily
\[
(G_x)=(H).
\]
\end{proof}

%----------------------------------------------------
\section{Local Bifurcation for the Riesz Equivariant Degree Jump}
%----------------------------------------------------

Let \(\lambda^*\) be an isolated critical value of the Riesz problem, and choose \(\varepsilon>0\) such that
\[
[\lambda^*-\varepsilon,\lambda^*+\varepsilon]
\cap\Lambda_R
=
\{\lambda^*\}.
\]
Let \(B_\zeta(X)\subset X\) be a sufficiently small \(G_R\)-invariant ball centered at the origin. The local equivariant degree jump is defined by
\[
\omega_R[\lambda^*,0]
=G_R\text{-}\deg
\bigl(
\Psi_R(\lambda^*+\varepsilon,\cdot),
B_\zeta(X)\bigr)-G_R\text{-}\deg
\bigl(
\Psi_R(\lambda^*-\varepsilon,\cdot),
B_\zeta(X)\bigr).
\]
For a negative eigenvalue
\(\mu\in\sigma_-(T_{R,j,m}(\lambda))\), let
\[
\mathfrak m_{j,m}(\mu;\lambda)
\]
denote the multiplicity of the irreducible \(G_R\)-representation
\(\mathcal V_{j,m}\) in the \(\mu\)-eigenspace of
\(T_{R,j,m}(\lambda)\); equivalently,
\(\mathfrak m_{j,m}(\mu;\lambda)\) is the number of copies of
\(\mathcal V_{j,m}\) contained in that eigenspace.

Let
\[
\lambda^{*-}:=\lambda^*-\varepsilon,
\qquad
\lambda^{*+}:=\lambda^*+\varepsilon,
\]
where \(\varepsilon>0\) is sufficiently small so that
\(\lambda^*\) is the only critical parameter in
\([\lambda^{*-},\lambda^{*+}]\). By the linearization property of the equivariant degree, one has
\[
\omega_R[\lambda^*,0]=
G_R\text{-}\deg\bigl(T_R(\lambda^{*+}),B_\zeta(X)
\bigr)-G_R\text{-}\deg\bigl(T_R(\lambda^{*-}),B_\zeta(X)
\bigr).
\]

Using the Fourier--isotypical decomposition and the multiplication property
of the equivariant degree, the two degree terms admit the representations
\begin{equation}\label{eq:riesz-degree-product}
\begin{aligned}
\omega_R[\lambda^*,0]
&=\prod_{m\geq0}
\prod_{j=0}^{r}
\prod_{\mu\in\sigma_-\left(
T_{R,j,m}(\lambda^{*+})
\right)}
\left(\deg_{\mathcal V_{j,m}}^{\,G_R}
\right)^{\mathfrak m_{j,m}
\left(\mu;\lambda^{*+}
\right)}
\\
&\quad-
\prod_{m\geq0}\prod_{j=0}^{r}
\prod_{\mu\in\sigma_-\left(
T_{R,j,m}(\lambda^{*-})
\right)}
\left(\deg_{\mathcal V_{j,m}}^{\,G_R}
\right)^{\mathfrak m_{j,m}
\left(\mu;\lambda^{*-}\right)}.
\end{aligned}
\end{equation}
The products are taken in the Burnside
ring \(A(G_R)\).

%----------------------------------------------------
\subsection{Local bifurcation for the Riesz problem}
%----------------------------------------------------
\begin{theorem}[Local bifurcation for the Riesz problem]
\label{thm:riesz-local-bifurcation}
Assume that \(g\) is sufficiently smooth and \(G_R\)-equivariant, that
\[
g(\lambda,0)=0,
\]
and that \(\lambda^*\) is an isolated critical value. If
\[
\omega_R[\lambda^*,0]\neq0
\qquad\text{in }A(G_R),
\]
then
\[
(\lambda^*,0)
\]
is a bifurcation point of nontrivial \(2\pi\)-periodic solutions of the Riesz equation. More precisely, every neighborhood of
\[
(\lambda^*,0)
\]
contains a solution
\[
(\lambda,x)
\]
such that
\[
x\neq0,
\qquad
\Psi_R(\lambda,x)=0.
\]

Moreover, let \((H)\) be a maximal orbit type occurring with nonzero
coefficient in
\[
\omega_R[\lambda^*,0].
\]
Then there exist nontrivial bifurcating periodic solutions whose isotropy group contains a subgroup conjugate to \(H\). In particular, if \((H)\) is maximal among the admissible orbit types in the corresponding critical
representation, then the bifurcating solutions detected by this coefficient have isotropy type \((H)\).
\end{theorem}

\begin{proof}
Suppose, to the contrary, that
\[
(\lambda^*,0)
\]
is not a bifurcation point. Then, after reducing \(\varepsilon>0\) and the radius \(\zeta>0\) if necessary, the trivial solution is the only zero of
\[
\Psi_R(\lambda,\cdot)
\]
in \(B_\zeta(X)\) for all
\[
\lambda\in[\lambda^*-\varepsilon,\lambda^*+\varepsilon].
\]
Hence the family
\[
\Psi_R(\lambda,\cdot),
\qquad
\lambda\in[\lambda^*-\varepsilon,\lambda^*+\varepsilon],
\]
defines an admissible \(G_R\)-equivariant homotopy on \(B_\zeta(X)\). By homotopy invariance of the \(G_R\)-equivariant degree,
\[
G_R\text{-}\deg
\bigl(
\Psi_R(\lambda^*+\varepsilon,\cdot),
B_\zeta(X)
\bigr)
=
G_R\text{-}\deg
\bigl(
\Psi_R(\lambda^*-\varepsilon,\cdot),
B_\zeta(X)
\bigr).
\]
Therefore,
\[
\omega_R[\lambda^*,0]=0,
\]
contradicting the hypothesis. Thus, nontrivial solutions occur arbitrarily close to \((\lambda^*,0)\).

Now let \((H)\) be a maximal orbit type occurring with nonzero coefficient in \(\omega_R[\lambda^*,0]\). By the existence property of the \(G_R\)-equivariant degree, there exists a nontrivial bifurcating solution \(x\) whose isotropy group contains a subgroup conjugate to \(H\);
equivalently,
\[
(H)\leq (G_x).
\]
If \((H)\) is maximal among the admissible orbit types in the corresponding critical representation, then necessarily
\[
(G_x)=(H).
\]
\end{proof}

%----------------------------------------------------
\subsection{Comparison of the two local invariants}
%----------------------------------------------------

The Weyl and Riesz local invariants have parallel representation-theoretic structures but arise from different spectral mechanisms. For the Weyl problem, the invariant has the additive form
\[
\omega_W=
\sum_{(j,k,m)}d_{j,k,m}\rho_{j,k,m}^{W}\deg_{\mathcal V_{j,m}}^{\,G_W},
\]
where the winding numbers measure the local orientation of complex eigenvalue maps in the two-dimensional parameter plane.

For the Riesz problem, the invariant is obtained from the difference
\[
\omega_R=\deg_{G_R}^{+}-\deg_{G_R}^{-},
\]
where the two degrees are computed on opposite sides of a real critical parameter. The change is determined by the basic equivariant degrees of the representations whose real eigenvalues cross zero.

Thus, a nonzero winding contribution produces local bifurcation in the \(\mathcal G\times\bz_2\times S^1\)-equivariant Weyl problem, whereas a nonzero equivariant degree jump produces local bifurcation in the \(\mathcal G\times\bz_2\times O(2)\)-equivariant Riesz problem.

%----------------------------------------------------
\section{Global Bifurcation}
%----------------------------------------------------

Because the Weyl problem depends on two real parameters, its global continuation is formulated directly in the two-dimensional parameter plane using the twisted equivariant Rabinowitz alternative. The Riesz problem is a one-parameter problem and admits the usual equivariant global bifurcation alternative.

For the Weyl problem, define
\[
\mathcal S_W
=
\left\{
(\alpha,\beta,x)\in\mathbb R^2\times X:
\Psi_W(\alpha,\beta,x)=0
\right\},
\qquad
\mathcal S_W^*
=
\left\{
(\alpha,\beta,x)\in\mathcal S_W:
x\neq0
\right\},
\]
and
\[
\mathcal T_W=\mathbb R^2\times\{0\}.
\]
For the Riesz problem, define
\[
\mathcal S_R=
\left\{
(\lambda,x)\in\mathbb R\times X:
\Psi_R(\lambda,x)=0
\right\},
\qquad
\mathcal S_R^*
=
\left\{
(\lambda,x)\in\mathcal S_R:
x\neq0
\right\},
\]
and
\[
\mathcal T_R=\mathbb R\times\{0\}.
\]
A global bifurcating component is a connected component of the closure of the corresponding nontrivial solution set that meets the trivial branch; see, for example, \cite{Kur,Rabinowitz1971}.
%----------------------------------------------------
\subsection{Global bifurcation for the Weyl problem}
%----------------------------------------------------

For an isolated critical point
\[
p^*=(\alpha^*,\beta^*)\in\Lambda_W,
\]
a nonempty subset
\[
\mathcal C\subset\mathcal S_W^*
\]
is called a branch of nontrivial solutions if there exists a connected
component
\[
\mathcal D
\subset
\overline{\mathcal S_W^*}
\]
such that
\[
\mathcal C=\mathcal S_W^*\cap\mathcal D.
\]
If
\[
(p^*,0)\in\mathcal D\cap\mathcal T_W,
\]
then \((p^*,0)\) is called a branching point of \(\mathcal C\).

\begin{theorem}[Equivariant Rabinowitz alternative for the Weyl problem]
\label{thm:weyl-rabinowitz-alternative}
Let
\[
U\subset\mathbb R^2\times X
\]
be an open bounded \(G_W\)-invariant set such that
\[
\partial U\cap
\bigl(\Lambda_W\times\{0\}\bigr)
=
\varnothing.
\]
Suppose that \(\mathcal C\subset\mathcal S_W^*\) is a branch of
nontrivial solutions bifurcating from an isolated critical point
\[
(p^*,0)\in
U\cap
\bigl(\Lambda_W\times\{0\}\bigr).
\]
Then at least one of the following alternatives occurs:
\begin{enumerate}
\item
\[
\mathcal C\cap\partial U\neq\varnothing;
\]

\item the set
\[
\overline{\mathcal C}
\cap
\bigl(\Lambda_W\times\{0\}\bigr)
=
\left\{
(p_0,0),\ldots,(p_N,0)
\right\}
\]
is finite, with \(p_0=p^*\), and
\[
\sum_{\ell=0}^{N}
\omega_W[p_\ell,0]
=
0
\qquad
\text{in }A_1^t(G_W).
\]
\end{enumerate}
\end{theorem}

The result follows from the Rabinowitz alternative for the twisted \(G_W\)-equivariant degree applied to the two-parameter compact perturbation
\[
\Psi_W:\mathbb R^2\times X\longrightarrow X.
\]
The local contribution of each isolated critical point is precisely the bifurcation invariant
\[
\omega_W[p,0]
=
G_W\text{-}\deg
\bigl(
\mathfrak F_W,\Omega_W
\bigr).
\]
If the branch does not meet the boundary of \(U\), additivity,
excision, and homotopy invariance of the twisted equivariant degree imply that the total bifurcation index of all critical points of the trivial branch met by the component must vanish. Hence
\[
\sum_{\ell=0}^{N}
\omega_W[p_\ell,0]
=
0.
\]

%----------------------------------------------------
\subsection{Global bifurcation for the Riesz problem}
%----------------------------------------------------

Let \(\lambda^*\) be an isolated critical value and let
\(\mathcal C_R\) be the connected component of \(
\overline{\mathcal S_R^*}\)
containing \((\lambda^*,0)\).

\begin{theorem}[Equivariant Rabinowitz alternative for the Riesz problem]
\label{thm:riesz-global-bifurcation}
Assume that \(\Psi_R\) is a \(G_R\)-equivariant compact perturbation of the
identity and that
\[
\omega_R[\lambda^*,0]\neq0
\]
in \(A(G_R)\). Then \(\mathcal C_R\) is a nontrivial continuum of periodic
solutions containing \((\lambda^*,0)\). Moreover, either
\begin{enumerate}
\item \(\mathcal C_R\) is unbounded in \(\mathbb R\times X\), or
\item \(\mathcal C_R\) meets the trivial branch at another critical point
      \((\lambda_1,0)\), where \(\lambda_1\neq\lambda^*\).
\end{enumerate}
\end{theorem}

This is the equivariant form of the Rabinowitz global bifurcation alternative. Its proof follows from the homotopy invariance and excision properties of the \(G_R\)-equivariant degree.

If a bounded global component meets the trivial branch at only finitely many critical values
\[
\lambda_1,\ldots,\lambda_N,
\]
then the corresponding bifurcation invariants satisfy
\[
\sum_{k=1}^{N}\omega_R[\lambda_k,0]=0
\qquad\text{in }A(G_R).
\]

%----------------------------------------------------
\subsection{Comparison of the global alternatives}
%----------------------------------------------------

The global continuation mechanisms for the Weyl and Riesz problems are parallel, although their parameter spaces are different. For the one-sided Weyl equation, the critical points lie in the two-dimensional parameter plane, and the global continuation is formulated directly in \(\mathbb R^2\times X\) through the twisted equivariant Rabinowitz alternative. Accordingly, a bounded global branch can meet only finitely many critical points, and the sum of the corresponding local Weyl bifurcation invariants must vanish. If such a cancellation is impossible, the branch is necessarily unbounded.

For the Riesz equation, the parameter space is one-dimensional. The ordinary equivariant degree jump directly yields a global continuum whose alternatives are unboundedness or return to another critical value of the trivial branch.

\begin{remark}
Theorems~\ref{thm:weyl-rabinowitz-alternative}
and~\ref{thm:riesz-global-bifurcation} show that a connected component detected by a nonzero equivariant bifurcation invariant cannot terminate inside a bounded neighborhood of the corresponding bifurcation point. If, in addition, the solution set is uniformly bounded in \(X\) whenever
the associated parameters remain in a bounded set, then any unboundedness of the component must occur in the parameter direction.

For the Weyl problem, this means that the parameter pair
\((\alpha,\beta)\) must leave every bounded subset of \(\mathbb R^2\), unless the component meets the trivial branch at other critical points whose local Weyl bifurcation invariants cancel in the sense of the global Rabinowitz alternative. For the Riesz problem, the parameter \(\lambda\) must become unbounded unless the component returns to the trivial branch at a different critical value.
\end{remark}

\section{Examples}
\subsection{An $S_4\times\bz_2\times S^1$-Equivariant One-Sided Weyl Example}

Let
\[
C_4=
\begin{pmatrix}
3&-1&-1&-1\\
-1&3&-1&-1\\
-1&-1&3&-1\\
-1&-1&-1&3
\end{pmatrix}
\]
and define the $8\times 8$ matrix
\[
\mathcal C_{\beta,\xi}
=
\begin{pmatrix}
\xi C_4&-\beta I_4\\
\beta I_4&\xi C_4
\end{pmatrix}.
\]
We consider the $2\pi$-periodic problem
\begin{equation}\label{eq:S4-one-sided-Weyl} D_{W,+}^{q}x(t) = \alpha x(t) +\mathcal C_{\beta,\xi}x(t) -\kappa \frac{\lVert x(t)\rVert^2} {1+\lVert x(t)\rVert^2}\,x(t), \qquad x(t)\in\mathbb R^8. 
\end{equation}
where
\[
\kappa, \xi>0,
\qquad
\alpha,\beta\in\mathbb R.
\]

Write
\[
x=\begin{pmatrix}
u\\v
\end{pmatrix},
\qquad
u,v\in\mathbb R^4.
\]
Then \eqref{eq:S4-one-sided-Weyl} is equivalent to
\[
\begin{aligned}
D_{W,+}^{q}u
&=
\alpha u+\xi C_4u-\beta v
-\kappa
\frac{\lVert u\rVert^2+\lVert v\rVert^2}
{1+\lVert u\rVert^2+\lVert v\rVert^2}\,u,\\
D_{W,+}^{q}v
&=
\alpha v+\xi C_4v+\beta u
-\kappa
\frac{\lVert u\rVert^2+\lVert v\rVert^2}
{1+\lVert u\rVert^2+\lVert v\rVert^2}\,v.
\end{aligned}
\]

For each \(s\in S_4\), let \(P_s\) denote the corresponding permutation matrix acting on \(\mathbb R^4\). The induced action of \(S_4\) on
\[
\mathbb R^8=\mathbb R^4\oplus\mathbb R^4
\]
is represented by
\[
\Pi(s):=
\begin{pmatrix}
P_s&0\\
0&P_s
\end{pmatrix}.
\]
Since
\[
C_4P_s=P_sC_4,
\]
it follows that
\[
\mathcal C_{\beta,\xi}\Pi(s)=
\Pi(s)\mathcal C_{\beta,\xi}.
\]
Hence, \(\mathcal C_{\beta,\xi}\) commutes with the induced \(S_4\)-action
on \(\mathbb R^8\).

The $\bz_2$-action is
\[
x\longmapsto -x,
\]
and the $S^1$-action is given by time translation,
\[
(\theta x)(t)=x(t+\theta).
\]
Consequently, \eqref{eq:S4-one-sided-Weyl} is equivariant under
\[
G_W=S_4\times\bz_2\times S^1.
\]

The natural permutation representation of $S_4$ on $\mathbb R^4$
decomposes as
\[
\mathbb R^4=V_0\oplus V_1,
\]
where
\[
V_0=\operatorname{span}\{(1,1,1,1)^T\}
\]
is the trivial representation and
\[
V_1=\left\{w\in\mathbb R^4:
w_1+w_2+w_3+w_4=0\right\}
\]
is the three-dimensional standard representation.

Thus the spatial eigenvalues of $C_4$ are
\[
\ell_0=0
\]
with multiplicity one and
\[
\ell_1=4
\]
with multiplicity three.

Accordingly,
\[
\mathbb R^8
=(V_0\oplus V_0)
\oplus (V_1\oplus V_1).
\]

The matrix $\alpha I_8+\mathcal C_{\beta,\xi}$ reduces to
\[
\begin{pmatrix}
\alpha+\xi\ell_j&-\beta\\
\beta&\alpha+\xi\ell_j
\end{pmatrix}.
\]
Therefore, on the trivial spatial component,
\[
\eta_0^\pm=\alpha\pm i\beta,
\]
each with complex multiplicity one, while on the standard component,
\[
\eta_1^\pm=\alpha+4\xi\pm i\beta,
\]
each with complex multiplicity three.

Hence, by~\eqref{eq:critical}, for the trivial spatial component, the critical parameter values are
\[
\alpha_{0,m}=m^q\cos\left(\frac{\pi q}{2}\right),
\qquad
\beta_{0,m}=m^q\sin\left(\frac{\pi q}{2}\right).
\]
For the standard spatial component, they are
\[
\alpha_{1,m}=m^q\cos\left(\frac{\pi q}{2}\right)-4\xi,
\qquad
\beta_{1,m}=m^q\sin\left(\frac{\pi q}{2}\right).
\]

For example, let
\[
q=\frac12,
\qquad
\xi=\frac14,
\qquad
\kappa=1.
\]

Since
\[
m^q=\sqrt m
\]
and
\[
\cos\left(\frac{\pi q}{2}\right)
=
\sin\left(\frac{\pi q}{2}\right)
=
\frac1{\sqrt2},
\]
the critical parameter pair is
\[
\left(
\sqrt{\frac m2}-\frac{\ell_j}{4},
\sqrt{\frac m2}
\right).
\]
The critical parameter pairs for \(m=0,1,2,3\) are listed in Table~\ref{tab:weyl-critical-pairs}.

\begin{table}[H]
\centering
\caption{Critical parameter pairs for \(m=0,1,2,3\).}
\label{tab:weyl-critical-pairs}
\renewcommand{\arraystretch}{1.35}
\setlength{\tabcolsep}{10pt}
\begin{tabular}{ccccc}
\toprule
Temporal mode
& Spatial component
& \(\ell\)
& \(\alpha_{j,m}^*\)
& \(\beta_{j,m}^*\)\\
\midrule

\multirow{2}{*}{\(m=0\)}
& Trivial component \(V_0\)
& \(0\)
& \(0\)
& \(0\)\\

& Standard component \(V_1\)
& \(4\)
& \(-1\)
& \(0\)\\
\midrule

\multirow{2}{*}{\(m=1\)}
& Trivial component \(V_0\)
& \(0\)
& \(\dfrac{1}{\sqrt2}\)
& \(\dfrac{1}{\sqrt2}\)\\

& Standard component \(V_1\)
& \(4\)
& \(\dfrac{1}{\sqrt2}-1\)
& \(\dfrac{1}{\sqrt2}\)\\
\midrule

\multirow{2}{*}{\(m=2\)}
& Trivial component \(V_0\)
& \(0\)
& \(1\)
& \(1\)\\

& Standard component \(V_1\)
& \(4\)
& \(0\)
& \(1\)\\
\midrule

\multirow{2}{*}{\(m=3\)}
& Trivial component \(V_0\)
& \(0\)
& \(\sqrt{\dfrac32}\)
& \(\sqrt{\dfrac32}\)\\

& Standard component \(V_1\)
& \(4\)
& \(\sqrt{\dfrac32}-1\)
& \(\sqrt{\dfrac32}\)\\

\bottomrule
\end{tabular}
\end{table}

At each critical pair with \(m\geq0\), one has
\[
D_{(\alpha,\beta)}
\begin{pmatrix}
u_{j,m}\\[1mm]
v_{j,m}
\end{pmatrix}
(\alpha_{j,m}^*,\beta_{j,m}^*)
=
\begin{pmatrix}
-1&0\\
0&-1
\end{pmatrix}.
\]
Therefore,
\[
\operatorname{sgn}\det
D_{(\alpha,\beta)}
\begin{pmatrix}
u_{j,m}\\[1mm]
v_{j,m}
\end{pmatrix}
(\alpha_{j,m}^*,\beta_{j,m}^*)
=1.
\]
Thus every critical pair with \(m\geq1\) has winding number \(1\). 

For each critical parameter pair
\[
p_{j,m}^*
=
(\alpha_{j,m}^*,\beta_{j,m}^*),
\qquad m\geq1,
\]
the stationary block
\[
T_{W,0}(p_{j,m}^*)
=
T_W(p_{j,m}^*)\big|_{X_0}
\]
is an isomorphism. Moreover, in the present example,
\(T_{W,0}(p_{j,m}^*)\) belongs to the identity component of the space of
\(G_0\)-equivariant linear isomorphisms of \(X_0\). Hence it is
\(G_0\)-equivariantly homotopic to \(\Id_{X_0}\) through linear
isomorphisms.

By the homotopy invariance and normalization properties of the equivariant
degree,
\[
G_0\text{-}\deg
\left(
T_{W,0}(p_{j,m}^*),B_\zeta(X_0)
\right)
=
G_0\text{-}\deg
\left(
\Id_{X_0},B_\zeta(X_0)
\right)
=
\mathbf{1}_{A(G_0)}.
\]

Consequently, the stationary factor in the local Weyl bifurcation invariant
is the multiplicative identity and does not change the dynamic contribution.
Therefore,
\[
\omega_W[p_{j,m}^*,0]
=
G_W\text{-}\deg
\left(
\mathfrak L_W^+,\Omega_W^+
\right).
\]

Therefore, the corresponding dynamic contributions to the local Weyl
bifurcation invariant are
\[
\omega_W[p_{0,m}^*,0]
=\deg_{\cV_{0,m}}^{\,G_W}
\]
and
\[
\omega_W[p_{1,m}^*,0]
=\deg_{\cV_{1,m}}^{\,G_W}.
\]
Using GAP, we obtain the following local Weyl bifurcation invariants. The corresponding amalgamated subgroup notation and group-theoretic conventions are collected in the appendix; further details may be found in~\cite{AED,Dab1}.

\[
\begin{aligned}
\omega_W[p_{0,1}^*,0]
&=
\left(
\amal{S_4^p}{S_4}{}{\bz_1}{\bz_2}
\right),\\
\omega_W[p_{0,2}^*,0]
&=
\left(
\amal{S_4^p}{S_4}{}{\bz_2}{\bz_4}
\right),\\
\omega_W[p_{0,3}^*,0]
&=
\left(
\amal{S_4^p}{S_4}{}{\bz_3}{\bz_6}
\right).
\end{aligned}
\]

For the standard spatial component, we have
\[
\begin{aligned}
\omega_W[p_{1,1}^*,0]
={}&
-\left(
\amal{D_1^p}{D_1}{}{\bz_1}{\bz_2}
\right)
-\left(
\amal{\bz_2^p}{\bz_2^m}{}{\bz_1}{\bz_2}
\right)\\
&+
\left(
\amal{\bz_3^p}{\bz_1}{}{\bz_1}{\bz_6}
\right)
+
\left(
\amal{D_2^p}{D_2^m}{}{\bz_1}{\bz_2}
\right)\\
&+
\left(
\amal{\bz_4^p}{\bz_2^m}{}{\bz_1}{\bz_4}
\right)
+
\left(
\amal{D_3^p}{D_3}{}{\bz_1}{\bz_2}
\right)\\
&+
\left(
\amal{D_4^p}{D_4^z}{}{\bz_1}{\bz_2}
\right),
\end{aligned}
\]
\[
\begin{aligned}
\omega_W[p_{1,2}^*,0]
={}&
-\left(
\amal{D_1^p}{D_1}{}{\bz_2}{\bz_4}
\right)
-\left(
\amal{\bz_2^p}{\bz_2^m}{}{\bz_2}{\bz_4}
\right)\\
&+
\left(
\amal{\bz_3^p}{\bz_1}{}{\bz_2}{\bz_{12}}
\right)
+
\left(
\amal{D_2^p}{D_2^m}{}{\bz_2}{\bz_4}
\right)\\
&+
\left(
\amal{\bz_4^p}{\bz_2^m}{}{\bz_2}{\bz_8}
\right)
+
\left(
\amal{D_3^p}{D_3}{}{\bz_2}{\bz_4}
\right)\\
&+
\left(
\amal{D_4^p}{D_4^z}{}{\bz_2}{\bz_4}
\right),
\end{aligned}
\]
and
\[
\begin{aligned}
\omega_W[p_{1,3}^*,0]
={}&
-\left(
\amal{D_1^p}{D_1}{}{\bz_3}{\bz_6}
\right)
-\left(
\amal{\bz_2^p}{\bz_2^m}{}{\bz_3}{\bz_6}
\right)\\
&+
\left(
\amal{\bz_3^p}{\bz_1}{}{\bz_3}{\bz_{18}}
\right)
+
\left(
\amal{D_2^p}{D_2^m}{}{\bz_3}{\bz_6}
\right)\\
&+
\left(
\amal{\bz_4^p}{\bz_2^m}{}{\bz_3}{\bz_{12}}
\right)
+
\left(
\amal{D_3^p}{D_3}{}{\bz_3}{\bz_6}
\right)\\
&+
\left(
\amal{D_4^p}{D_4^z}{}{\bz_3}{\bz_6}
\right).
\end{aligned}
\]

The maximal orbit types occurring in each local bifurcation invariant are summarized in Table~\ref{tab:maximal-orbit-types}.
\begin{table}[H]
\centering
\caption{Maximal orbit types associated with the critical parameter points.}
\label{tab:maximal-orbit-types}
\renewcommand{\arraystretch}{1.45}
\setlength{\tabcolsep}{8pt}

\begin{tabularx}{\textwidth}{
    >{\centering\arraybackslash}p{2.2cm}
    >{\centering\arraybackslash}X
}
\toprule
\textbf{Critical point} & \textbf{Maximal orbit types} \\
\midrule

$p_{0,1}^{*}$
&
$\left(
\amal{S_4^p}{S_4}{}{\bz_1}{\bz_2}
\right)$
\\
\midrule

$p_{0,2}^{*}$
&
$\left(
\amal{S_4^p}{S_4}{}{\bz_2}{\bz_4}
\right)$
\\
\midrule

$p_{0,3}^{*}$
&
$\left(
\amal{S_4^p}{S_4}{}{\bz_3}{\bz_6}
\right)$
\\
\midrule

$p_{1,1}^{*}$
&
\(\begin{aligned}
&\left(
\amal{\bz_3^p}{\bz_1}{}{\bz_1}{\bz_6}
\right)
\qquad
\left(
\amal{D_2^p}{D_2^m}{}{\bz_1}{\bz_2}
\right)
\qquad
\left(
\amal{\bz_4^p}{\bz_2^m}{}{\bz_1}{\bz_4}
\right)
\\
&\left(
\amal{D_3^p}{D_3}{}{\bz_1}{\bz_2}
\right)
\qquad
\left(
\amal{D_4^p}{D_4^z}{}{\bz_1}{\bz_2}
\right)
\end{aligned}\)
\\[2mm]
\midrule

$p_{1,2}^{*}$
&
\(\begin{aligned}
&\left(
\amal{\bz_3^p}{\bz_1}{}{\bz_2}{\bz_{12}}
\right)
\qquad
\left(
\amal{D_2^p}{D_2^m}{}{\bz_2}{\bz_4}
\right)
\qquad
\left(
\amal{\bz_4^p}{\bz_2^m}{}{\bz_2}{\bz_8}
\right)
\\
&\left(
\amal{D_3^p}{D_3}{}{\bz_2}{\bz_4}
\right)
\qquad
\left(
\amal{D_4^p}{D_4^z}{}{\bz_2}{\bz_4}
\right)
\end{aligned}\)
\\[2mm]
\midrule

$p_{1,3}^{*}$
&
\(\begin{aligned}
&\left(
\amal{\bz_3^p}{\bz_1}{}{\bz_3}{\bz_{18}}
\right)
\qquad
\left(
\amal{D_2^p}{D_2^m}{}{\bz_3}{\bz_6}
\right)
\qquad
\left(
\amal{\bz_4^p}{\bz_2^m}{}{\bz_3}{\bz_{12}}
\right)
\\
&\left(
\amal{D_3^p}{D_3}{}{\bz_3}{\bz_6}
\right)
\qquad
\left(
\amal{D_4^p}{D_4^z}{}{\bz_3}{\bz_6}
\right)
\end{aligned}\)
\\

\bottomrule
\end{tabularx}
\end{table}
The one-sided Weyl derivative was approximated by a periodic Grünwald--Letnikov scheme using \(N=64\) grid points on \([0,2\pi]\). Hence, the step size was
\[
h=\frac{2\pi}{64}=\frac{\pi}{32}\approx0.0982.
\]
The resulting nonlinear algebraic system was solved by the Broyden method. A representative numerical periodic solution is shown in Figure~\ref{fig:S4-Weyl-periodic-solution}, while the corresponding bifurcation diagram is presented in Figure~\ref{fig:S4-Weyl-Hq-bifurcation}.

\begin{figure}[H]
    \centering
    \includegraphics[width=0.72\textwidth]{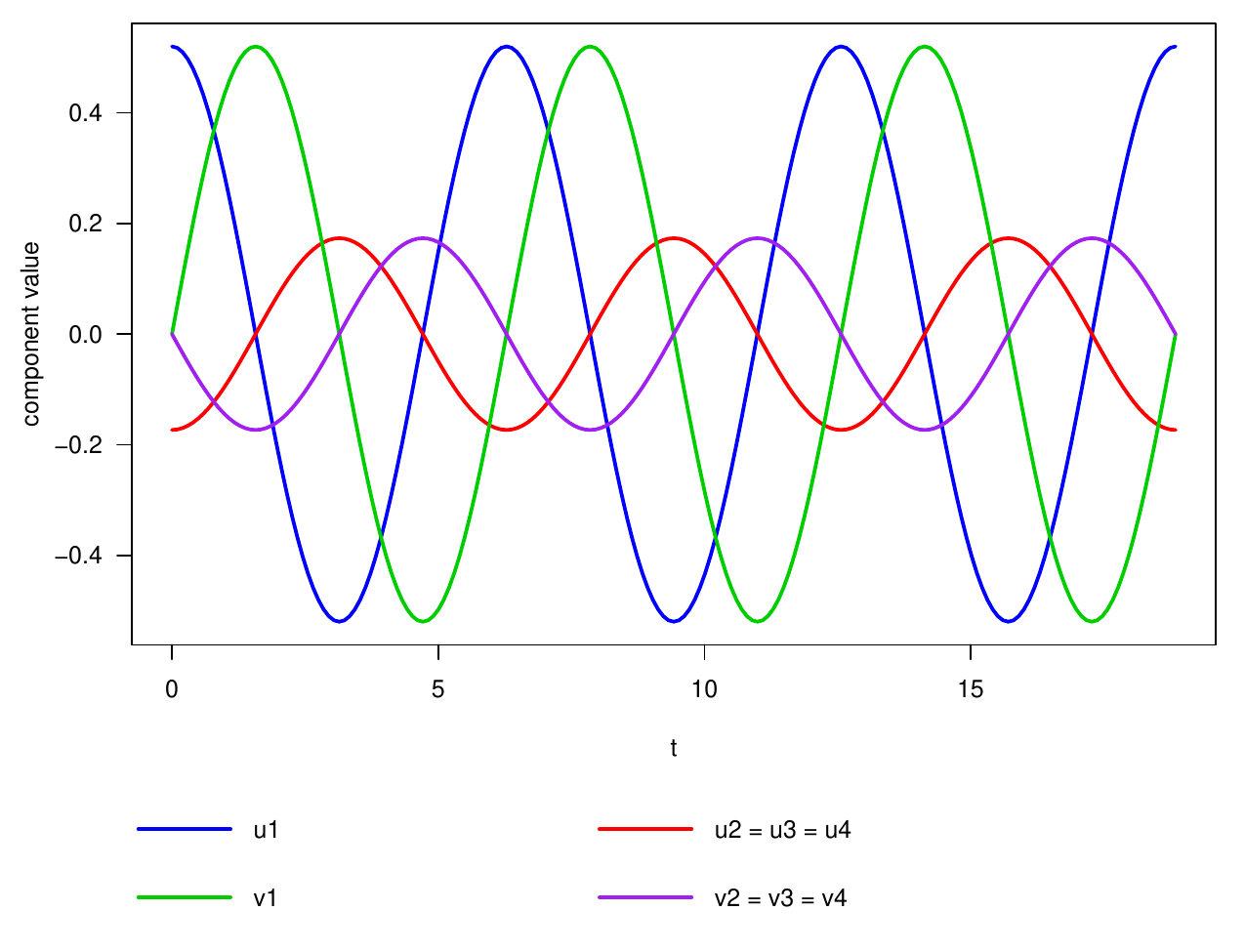}
    \caption{Numerical periodic solution of the eight-dimensional
    \(S_4\)-equivariant Weyl system. The components occur in
    symmetry-related pairs.}
    \label{fig:S4-Weyl-periodic-solution}
\end{figure}

\begin{figure}[H]
    \centering
    \includegraphics[width=0.72\textwidth]{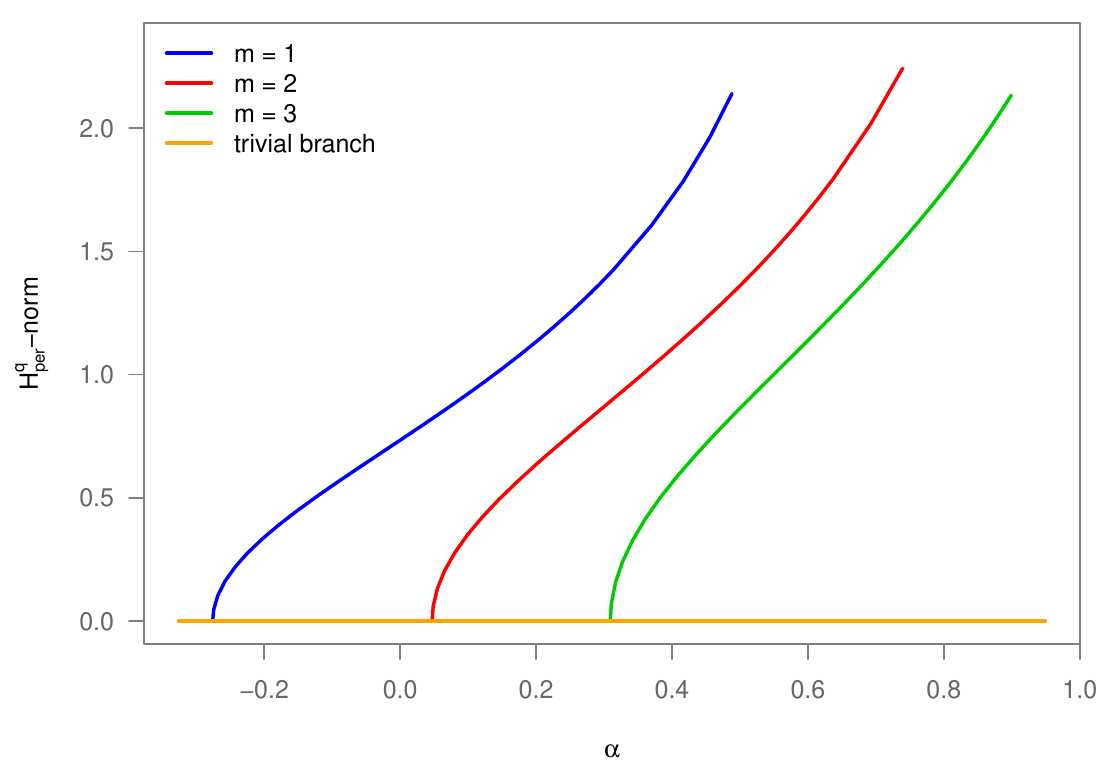}
\caption{\(H_{\mathrm{per}}^{q}\)-norm bifurcation diagram for the \(S_4\)-equivariant Weyl system, showing the trivial branch and the nontrivial branches for \(m=1,2,3\).}
    \label{fig:S4-Weyl-Hq-bifurcation}
\end{figure}

%----------------------------------------------------
\subsection{Periodic Riesz equation with \(S_4\)-symmetry}
%----------------------------------------------------

Consider
\[
(-\partial_t^2)^{1/4}x(t)
=\lambda x(t)
+\frac14 C_4x(t)-
\frac{\lVert x(t)\rVert^2}
{1+\lVert x(t)\rVert^2}\,x(t),
\qquad x(t)\in\mathbb R^4.
\]

where
\[
C_4=
\begin{pmatrix}
3&-1&-1&-1\\
-1&3&-1&-1\\
-1&-1&3&-1\\
-1&-1&-1&3
\end{pmatrix}.
\]
As in the preceding example, \(S_4\) acts by coordinate permutations,
\(\bz_2\) acts by \(x\mapsto-x\), and \(O(2)\) acts on the temporal
variable. Hence,
\[
G_R=S_4\times \bz_2\times O(2).
\]

Recall that
\[
\mathbb R^4=V_0\oplus W,
\qquad
C_4|_{V_0}=0,
\qquad
C_4|_W=4I_W.
\]
Let
\[
\ell_0=0,
\qquad
\ell_1=4
\]
denote the corresponding spatial eigenvalues. Since
\[
(-\partial_t^2)^{1/4}e^{imt}
=
|m|^{1/2}e^{imt},
\]
the characteristic value on the spatial-temporal block
\((j,m)\) is
\[
\Delta_{j,m}^{R}(\lambda)
=
|m|^{1/2}-\lambda-\frac{\ell_j}{4}.
\]
Therefore, the critical parameter values are
\[
\lambda_{j,m}^{*}
=|m|^{1/2}-\frac{\ell_j}{4}.
\]
In particular,
\[
\lambda_{0,m}^{*}=\sqrt{|m|}
\]
on the trivial spatial component and
\[
\lambda_{1,m}^{*}=\sqrt{|m|}-1
\]
on the standard component \(W\).

For the first three nonzero temporal modes, the critical values on \(W\) are
\[
\lambda_{1,1}^{*}=0,
\qquad
\lambda_{1,2}^{*}=\sqrt2-1,
\qquad
\lambda_{1,3}^{*}=\sqrt3-1.
\]
For \(m=0\), one obtains
\[
\lambda_{0,0}^{*}=0,
\qquad
\lambda_{1,0}^{*}=-1,
\]
which correspond to stationary spatial modes.
Let
\[
\lambda_{1,2}^{-}<\lambda_{1,2}^{*}
<\lambda_{1,2}^{+},
\qquad
\lambda_{1,2}^{*}=\sqrt2-1,
\]
be sufficiently close to \(\lambda_{1,2}^{*}\), so that no other critical
value lies in
\([\lambda_{1,2}^{-},\lambda_{1,2}^{+}]\).
The signs of the eigenvalues of
\(D_x\Psi_R(\lambda_{1,2}^{-},0)\) and
\(D_x\Psi_R(\lambda_{1,2}^{+},0)\) on the corresponding
spatio-temporal components are displayed in
Tables~\ref{tab:Riesz-lambda12-minus} and
\ref{tab:Riesz-lambda12-plus}.

\begin{table}[H]
\centering
\footnotesize

\begin{minipage}[t]{0.47\textwidth}
\centering
\caption{Signs at \(\lambda_{1,2}^{-}\)}
\label{tab:Riesz-lambda12-minus}
\begin{tabularx}{\textwidth}{lYY}
\toprule
\multicolumn{3}{c}{
\(\mu_{j,m}^{R}(\lambda_{1,2}^{-})\)
}\\
\midrule
\(m\setminus j\) & \(0\) & \(1\) \\
\midrule
\(0\)       & \(-\) & \(-\) \\
\(1\)       & \(+\) & \(-\) \\
\(2\)       & \(+\) & \(+\) \\
\(m\geq 3\) & \(+\) & \(+\) \\
\bottomrule
\end{tabularx}
\end{minipage}
\hfill
\begin{minipage}[t]{0.47\textwidth}
\centering
\caption{Signs at \(\lambda_{1,2}^{+}\)}
\label{tab:Riesz-lambda12-plus}
\begin{tabularx}{\textwidth}{lYY}
\toprule
\multicolumn{3}{c}{
\(\mu_{j,m}^{R}(\lambda_{1,2}^{+})\)
}\\
\midrule
\(m\setminus j\) & \(0\) & \(1\) \\
\midrule
\(0\)       & \(-\) & \(-\) \\
\(1\)       & \(+\) & \(-\) \\
\(2\)       & \(+\) & \(-\) \\
\(m\geq 3\) & \(+\) & \(+\) \\
\bottomrule
\end{tabularx}
\end{minipage}

\end{table}
\normalsize
By \eqref{eq:riesz-degree-product}, the equivariant degrees on the two sides of the
critical value
\[
\lambda_{1,2}^{*}=\sqrt{2}-1
\]
are given by
\[
\begin{aligned}
\eqdeg{G_R}\!\left(
D_x\Psi_R(\lambda_{1,2}^{-},0),B_\zeta(X)
\right)
={}&
\deg_{\mathcal V_{0,0}}^{G_R}
\cdot
\deg_{\mathcal V_{1,0}}^{G_R}
\cdot
\deg_{\mathcal V_{1,1}}^{G_R},
\\
\eqdeg{G_R}\!\left(
D_x\Psi_R(\lambda_{1,2}^{+},0),B_\zeta(X)
\right)
={}&
\deg_{\mathcal V_{0,0}}^{G_R}
\cdot
\deg_{\mathcal V_{1,0}}^{G_R}
\cdot
\deg_{\mathcal V_{1,1}}^{G_R}
\cdot
\deg_{\mathcal V_{1,2}}^{G_R},
\end{aligned}
\]
where
\[
G_R=S_4\times \bz_2\times O(2).
\]
Indeed, only the critical block \(\mathcal V_{1,2}\) changes sign as
\(\lambda\) passes through \(\lambda_{1,2}^{*}\).

After completing the computation, we obtain the following maximal orbit types. The full list of orbit-type contributions is provided in the appendix.
\[
\left(\amal{D_3^p}{\bz_1}{}{\bz_2}{D_{12}}\right),
\left(\amal{D_4^p}{\bz_2^m}{}{\bz_2}{D_8}\right),
\left(\amal{D_2^p}{D_2^m}{}{D_2}{D_4}\right),
\left(\amal{D_3^p}{D_3}{}{D_2}{D_4}\right),
\left(\amal{D_4^p}{D_4^z}{}{D_2}{D_4}\right).
\]
Consequently, for each of these maximal orbit types, there exists at least one branch of nonconstant \(2\pi\)-periodic solutions bifurcating from the critical point \((\lambda_{1,2}^*,0)\).

\subsection{Comparison}

The Riesz equation and the one-sided Weyl equation differ mainly in their spectral structures. For the Riesz problem, the temporal multipliers are real and nonnegative, and the linearized blocks are self-adjoint. Hence, critical values are determined by a single real equation, and the equivariant degree can be computed from the signs of the characteristic eigenvalues.

For the one-sided Weyl problem, the multipliers \((im)^q\) are generally complex, so the linearized blocks are nonself-adjoint. A nonzero temporal mode becomes critical only when both the real and imaginary parts vanish. This naturally leads to a two-parameter formulation and a degree calculation based on winding numbers.

Thus, the Riesz example provides a simpler self-adjoint setting, whereas the Weyl example illustrates equivariant bifurcation for nonself-adjoint problems with complex characteristic values.

\section{Conclusion}\label{sec:conclusion}

In this paper, we developed an equivariant bifurcation framework for periodic fractional differential equations governed by the one-sided Weyl and Riesz fractional operators. Although both problems admit a common formulation as equivariant compact perturbations of the identity, their spectral structures lead to different bifurcation mechanisms. The Fourier and spatial isotypical decompositions reduce the linearized problems to finite-dimensional spatio-temporal blocks and identify the relevant critical modes and symmetries.

For the one-sided Weyl equation, the nonzero temporal Fourier multipliers are generally complex and the linearized operators are nonself-adjoint. Criticality therefore requires the simultaneous vanishing of the real and imaginary parts of a complex characteristic function, naturally leading to a two-parameter problem. The local invariant combines winding numbers with basic equivariant degrees and detects nonconstant periodic solutions together with their possible symmetries. The twisted equivariant Rabinowitz alternative then describes the global continuation of the resulting branches.

For the Riesz equation, the temporal multipliers are real and the relevant linearized blocks are self-adjoint. Criticality is determined by real spectral crossings in a one-dimensional parameter space, and the local invariant is given by the jump of the ordinary equivariant degree across an isolated critical value. A nonzero degree jump yields bifurcation of nontrivial periodic solutions, while the Burnside-ring components identify possible spatio-temporal isotropy types. The corresponding global alternative gives either unbounded continuation or a return to the trivial branch at another critical value.

The examples illustrate these two mechanisms in concrete symmetric fractional systems. In particular, the \(S_4\)-equivariant one-sided Weyl example demonstrates the role of complex characteristic values and winding numbers, while the Riesz examples exhibit the real spectral-crossing mechanism. The numerical computations for the Weyl system further illustrate periodic solutions associated with different temporal modes. Overall, the results provide a unified equivariant framework that preserves the essential distinction between the nonself-adjoint Weyl and self-adjoint Riesz settings.

Future work may extend the present framework to broader classes of nonlocal and fractional evolution equations, including systems with delay, distributed memory, or more general pseudo-differential operators. It would also be interesting to investigate higher-dimensional parameter families and to develop equivariant degree methods for more complicated symmetry groups and noncompact settings.

\appendix
\section{Notation for Amalgamated Subgroups}
\label{app:amalgamated}

For completeness, we describe the notation used for subgroups of a direct
product. Let \(G_1\) and \(G_2\) be groups, and choose subgroups
\[
K_1\leq G_1,
\qquad
K_2\leq G_2.
\]
Suppose that \(L\) is a group for which there exist surjective homomorphisms
\[
\varphi:K_1\to L,
\qquad
\psi:K_2\to L.
\]
The corresponding fiber product is defined by
\[
K_1\,{}^{\varphi}\!\times_L^{\psi}K_2
:=
\left\{
(g_1,g_2)\in K_1\times K_2:
\varphi(g_1)=\psi(g_2)
\right\}.
\]
We refer to this subgroup as the amalgamated product of \(K_1\) and \(K_2\)
over \(L\).

In other words, only those pairs whose components determine the same element
of the common quotient \(L\) are retained. Consequently, this subgroup is
typically a proper subgroup of \(K_1\times K_2\).

Let
\[
N_1:=\ker\varphi,
\qquad
N_2:=\ker\psi.
\]
Then
\[
K_1/N_1\simeq L\simeq K_2/N_2,
\]
and the same subgroup can be written in the equivalent form
\[
K_1\,{}^{N_1}\!\times_L^{N_2}K_2
=
\left\{
(g_1,g_2)\in K_1\times K_2:
\varphi(g_1)=\psi(g_2)
\right\}.
\]
Equivalently, the cosets \(g_1N_1\) and \(g_2N_2\) correspond under the chosen identifications with \(L\). By Goursat's lemma, every subgroup of a direct product admits such a description after suitable choices of projections, kernels, and quotient maps.

In the applications considered here, the first factor is generally a finite symmetry group \(\mathcal G\), whereas the second factor is a subgroup of \(O(2)\). We write
\[
\bigl(\amal{K}{Z}{L}{R}{Q}\bigr),
\]
where \(K\leq\mathcal G\), \(Q\leq O(2)\), and \(Z\) denotes the kernel of the epimorphism from \(K\) onto \(L\). If \(L\) is dihedral, \(R\) specifies the preimage in \(Q\) of the rotation subgroup of \(L\). 
\section{Subgroups of \(S_4\times\bz_2\)}
\label{app:S4Z2-subgroups}
Let
\[
G_1=S_4\times \bz_2.
\]
We realize \(G_1\) as a permutation group on
\(\{1,2,3,4,5,6\}\), where \(S_4\) acts on
\(\{1,2,3,4\}\) and the nontrivial element of the second factor is
\[
\tau:=(5,6).
\]
Thus, a permutation \(\sigma\in S_4\) represents \((\sigma,1)\), while
\(\sigma\tau\) represents \((\sigma,-1)\).

We use \(\bz_n\) for the cyclic group of order \(n\), \(V_4\) for the Klein
four group, and \(D_n\) for the dihedral group of order \(2n\). In
particular,
\[
D_1\cong \bz_2,\qquad D_2\cong V_4,\qquad D_3\cong S_3.
\]
The suffixes in the subgroup abbreviations distinguish nonconjugate
embeddings of the same abstract group in \(G_1\):
\[
p=\text{product with the external } \bz_2,\qquad
z=\text{signed or diagonal reflection type},
\]
\[
m=\text{mixed embedding},\qquad
d=\text{diagonal embedding},\qquad
hd=\text{hybrid-diagonal embedding}.
\]
These suffixes are labels for embeddings and do not change the abstract
isomorphism type.

For a subgroup \(H\leq G_1\), GAP writes \(H^{G_1}\) for its conjugacy
class. The command \texttt{ConjugacyClassesSubgroups(G1)} returns the following \(33\) conjugacy classes of subgroups in Table~\ref{tab:S4Z2-subgroups}.

\small
\begin{longtable}{c c p{7.2cm} c p{3.2cm}}
\caption{Conjugacy-class representatives of subgroups of
\(S_4\times\bz_2\).}
\label{tab:S4Z2-subgroups}\\
\toprule
No. & Label & Representative & Order & Abstract type\\
\midrule
\endfirsthead
\toprule
No. & Label & Representative & Order & Abstract type\\
\midrule
\endhead
1 & $\bz_1$ & $\langle e\rangle$ & $1$ & $Z\bz_1$ \\
2 & $\bz_2$ & $\left\langle (1\,3)(2\,4)\right\rangle$ & $2$ & $\bz_2$ \\
3 & $D_1^z$ & $\left\langle (3\,4)(5\,6)\right\rangle$ & $2$ & $\bz_2$ \\
4 & $D_1$ & $\left\langle (3\,4)\right\rangle$ & $2$ & $\bz_2$ \\
5 & $\bz_2^m$ & $\left\langle (1\,3)(2\,4)(5\,6)\right\rangle$ & $2$ & $\bz_2$ \\
6 & $\bz_1^p$ & $\left\langle (5\,6)\right\rangle$ & $2$ & $\bz_2$ \\
7 & $\bz_3$ & $\left\langle (2\,4\,3)\right\rangle$ & $3$ & $\bz_3$ \\
8 & $\bz_2^p$ & $\left\langle (5\,6),\, (1\,3)(2\,4)\right\rangle$ & $4$ & $\bz_2\times \bz_2$ \\
9 & $D_2^z$ & $\left\langle (1\,4)(2\,3)(5\,6),\, (1\,3)(2\,4)\right\rangle$ & $4$ & $\bz_2\times \bz_2$ \\
10 & $D_2$ & $\left\langle (3\,4),\, (1\,2)(3\,4)\right\rangle$ & $4$ & $D_2$ \\
11 & $\bz_4$ & $\left\langle (1\,3\,2\,4),\, (1\,2)(3\,4)\right\rangle$ & $4$ & $\bz_4$ \\
12 & $V_4$ & $\left\langle (1\,4)(2\,3),\, (1\,3)(2\,4)\right\rangle$ & $4$ & $V_4$ \\
13 & $D_2^d$ & $\left\langle (3\,4)(5\,6),\, (1\,2)(3\,4)\right\rangle$ & $4$ & $D_2$ \\
14 & $\bz_4^d$ & $\left\langle (1\,3\,2\,4)(5\,6),\, (1\,2)(3\,4)\right\rangle$ & $4$ & $\bz_4$ \\
15 & $D_2^m$ & $\left\langle (1\,2)(3\,4)(5\,6),\, (3\,4)\right\rangle$ & $4$ & $D_2$ \\
16 & $D_1^p$ & $\left\langle (5\,6),\, (3\,4)\right\rangle$ & $4$ & $D_1\times \bz_2$ \\
17 & $\bz_3^p$ & $\left\langle (5\,6),\, (2\,4\,3)\right\rangle$ & $6$ & $\bz_3\times \bz_2\cong \bz_6$ \\
18 & $D_3$ & $\left\langle (3\,4),\, (2\,4\,3)\right\rangle$ & $6$ & $D_3$ \\
19 & $D_3^z$ & $\left\langle (3\,4)(5\,6),\, (2\,4\,3)\right\rangle$ & $6$ & $D_3$ \\
20 & $V_4^p$ & $\left\langle (1\,4)(2\,3),\, (1\,3)(2\,4),\, (5\,6)\right\rangle$ & $8$ & $V_4\times \bz_2$ \\
21 & $D_4^z$ & $\left\langle (1\,3)(2\,4)(5\,6),\, (3\,4),\, (1\,2)(3\,4)\right\rangle$ & $8$ & $D_4$ \\
22 & $\bz_4^p$ & $\left\langle (5\,6),\, (1\,3\,2\,4),\, (1\,2)(3\,4)\right\rangle$ & $8$ & $\bz_4\times \bz_2$ \\
23 & $D_4$ & $\left\langle (1\,4)(2\,3),\, (1\,3)(2\,4),\, (3\,4)\right\rangle$ & $8$ & $D_4$ \\
24 & $D_2^p$ & $\left\langle (5\,6),\, (3\,4),\, (1\,2)(3\,4)\right\rangle$ & $8$ & $D_2\times \bz_2$ \\
25 & $D_4^m$ & $\left\langle (1\,3)(2\,4)(5\,6),\, (1\,3\,2\,4),\, (1\,2)(3\,4)\right\rangle$ & $8$ & $D_4$ \\
26 & $D_4^{hd}$ & $\left\langle (1\,4)(2\,3),\, (1\,3)(2\,4),\, (3\,4)(5\,6)\right\rangle$ & $8$ & $D_4$ \\
27 & $D_3^p$ & $\left\langle (5\,6),\, (3\,4),\, (2\,4\,3)\right\rangle$ & $12$ & $D_3\times \bz_2$ \\
28 & $A_4$ & $\left\langle (1\,4)(2\,3),\, (1\,3)(2\,4),\, (2\,4\,3)\right\rangle$ & $12$ & $A_4$ \\
29 & $D_4^p$ & $\left\langle (1\,4)(2\,3),\, (1\,3)(2\,4),\, (5\,6),\, (3\,4)\right\rangle$ & $16$ & $D_4\times \bz_2$ \\
30 & $S_4$ & $\left\langle (1\,4)(2\,3),\, (1\,3)(2\,4),\, (2\,4\,3),\, (3\,4)\right\rangle$ & $24$ & $S_4$ \\
31 & $A_4^p$ & $\left\langle (1\,4)(2\,3),\, (1\,3)(2\,4),\, (2\,4\,3),\, (5\,6)\right\rangle$ & $24$ & $A_4\times \bz_2$ \\
32 & $S_4^m$ & $\left\langle (1\,4)(2\,3),\, (1\,3)(2\,4),\, (2\,4\,3),\, (3\,4)(5\,6)\right\rangle$ & $24$ & $S_4$ \\
33 & $S_4^p$ & $\left\langle (1\,4)(2\,3),\, (1\,3)(2\,4),\, (2\,4\,3),\, (5\,6),\, (3\,4)\right\rangle$ & $48$ & $S_4\times \bz_2$ \\
\bottomrule
\end{longtable}
\normalsize

\section{Computation of the Local Riesz Bifurcation Invariant}
The local equivariant bifurcation invariant at
\((\lambda_{1,2}^{*},0)\) is

\[
\begin{aligned}
\omega_R[\lambda_{1,2}^{*},0]
={}&
\eqdeg{G_R}\!\left(
D_x\Psi_R(\lambda_{1,2}^{+},0),B_\zeta(X)
\right)
-
\eqdeg{G_R}\!\left(
D_x\Psi_R(\lambda_{1,2}^{-},0),B_\zeta(X)
\right)
\\[1mm]
={}&
3\left(\amal{\bz_1^p}{\bz_1}{}{\bz_1}{D_1}\right)
-\left(\amal{\bz_2^m}{\bz_1}{}{\bz_1}{D_1}\right)
+3\left(\amal{\bz_2}{\bz_1}{}{\bz_1}{D_1}\right)
\\
&\quad
-\left(\bz_1\times D_1\right)
-\left(\amal{D_1^p}{\bz_1}{}{\bz_1}{D_2}\right)
-3\left(\amal{D_1^p}{D_1}{}{\bz_1}{D_1}\right)
\\
&\quad
-\left(\amal{D_1^p}{\bz_1}{}{\bz_1}{D_2}\right)
+\left(\amal{D_2^m}{\bz_1}{}{\bz_1}{D_2}\right)
-\left(\amal{D_2^d}{\bz_1}{}{\bz_1}{D_2}\right)
\\
&\quad
-\left(\amal{D_2}{\bz_1}{}{\bz_1}{D_2}\right)
-2\left(\amal{D_2}{D_1}{}{\bz_1}{D_1}\right)
-\left(\amal{\bz_2^p}{\bz_1}{}{\bz_1}{D_2}\right)
\\
&\quad
-2\left(\amal{\bz_2^p}{\bz_2^m}{}{\bz_1}{D_1}\right)
-\left(\amal{D_2^d}{\bz_1}{}{\bz_1}{D_2}\right)
-\left(\amal{\bz_2^p}{\bz_1}{}{\bz_1}{D_2}\right)
\\
&\quad
+\left(\amal{\bz_2^m}{\bz_1}{}{D_1}{D_2}\right)
+2\left(\bz_2^m\times D_1\right)
-\left(D_1\times D_1\right)
\\
&\quad
+\left(\amal{D_1^z}{\bz_1}{}{D_1}{D_2}\right)
+\left(\amal{D_3}{\bz_1}{}{\bz_1}{D_3}\right)
+\left(\amal{D_3^z}{\bz_1}{}{\bz_1}{D_3}\right)
\\
&\quad
+\left(\amal{D_2^p}{D_1}{}{\bz_1}{D_2}\right)
+\left(\amal{D_2^p}{\bz_2^m}{}{\bz_1}{D_2}\right)
+\left(\amal{D_2^p}{D_2^m}{}{\bz_1}{D_1}\right)
\\
&\quad
-\left(\amal{D_4}{\bz_1}{}{\bz_1}{D_4}\right)
+\left(\amal{D_4^m}{\bz_1}{}{\bz_1}{D_4}\right)
+\left(\amal{D_2^p}{D_1}{}{\bz_1}{D_2}\right)
\\
&\quad
+\left(\amal{V_4^p}{\bz_2^m}{}{\bz_1}{D_2}\right)
-\left(\amal{D_4}{D_2}{}{\bz_1}{D_1}\right)
-\left(\amal{D_2^m}{D_1}{}{D_1}{D_2}\right)
\\
&\quad
-\left(\amal{D_2^m}{\bz_2^m}{}{D_1}{D_2}\right)
-\left(D_2^m\times D_1 \right)
+\left(\amal{D_2}{D_1}{}{D_1}{D_2}\right)
\\
&\quad
-\left(\amal{D_2^z}{\bz_2^m}{}{D_1}{D_2}\right)
+\left(D_2\times D_1\right)
+\left(\amal{D_3^p}{D_3}{}{\bz_1}{D_1}\right)
\\
&\quad
+\left(D_3\times D_1\right)
+\left(\amal{D_4^p}{D_4^z}{}{\bz_1}{D_1}\right)
-\left(D_4^z\times D_1\right)\\
&-
2\left(\amal{\bz_2^m}{\bz_1}{}{\bz_2}{D_2}\right)
+\left(\amal{D_1}{\bz_1}{}{\bz_2}{D_2}\right)
+\left(\amal{D_1^z}{\bz_1}{}{\bz_2}{D_2}\right)
\\
&\quad
+\left(\amal{\bz_2}{\bz_1}{}{\bz_2}{D_2}\right)
+2\left(\bz_1\times D_2\right)
-\left(\amal{D_2^m}{\bz_2^m}{}{\bz_2}{D_2}\right)
\\
&\quad
-\left(\amal{D_2^m}{D_1}{}{\bz_2}{D_2}\right)
-\left(\amal{D_2^m}{\bz_1}{}{\bz_2}{D_4}\right)
-\left(\amal{D_2^z}{\bz_2^m}{}{\bz_2}{D_2}\right)
\\
&\quad
-\left(\amal{\bz_2^p}{\bz_1}{}{\bz_2}{D_4}\right)
-\left(\amal{V_4}{\bz_1}{}{\bz_2}{D_4}\right)
-\left(\amal{D_2}{\bz_1}{}{\bz_2}{D_4}\right)
\\
&\quad
-\left(\amal{D_2^z}{\bz_1}{}{\bz_2}{D_4}\right)
-\left(\amal{\bz_1^p}{\bz_1}{}{D_2}{D_4}\right)
-\left(\amal{\bz_2^m}{\bz_1}{}{D_2}{D_4}\right)
\\
&\quad
-2\left(\bz_2^m\times D_2\right)
-\left(D_1\times D_2\right)
-\left(\amal{D_1^z}{\bz_1}{}{D_2}{D_4}\right)
\\
&\quad
-\left(\amal{\bz_2}{\bz_1}{}{D_2}{D_4}\right)
+\left(\amal{D_2^p}{D_1}{}{\bz_2}{D_4}\right)
+\left(\amal{D_4}{\bz_1}{}{\bz_2}{D_8}\right)
\\
&\quad
+\left(\amal{D_4^z}{\bz_1}{}{\bz_2}{D_8}\right)
+\left(\amal{D_2^p}{\bz_2^m}{}{\bz_2}{D_4}\right)
+\left(\amal{V_4^p}{\bz_2^m}{}{\bz_2}{D_4}\right)
\\
&\quad
+2\left(\amal{D_1^p}{D_1}{}{D_2}{D_4}\right)
+\left(\amal{D_2^m}{D_1}{}{D_2}{D_4}\right)
+\left(\amal{D_2^m}{\bz_2^m}{}{D_2}{D_4}\right)
\\
&\quad
+\left(D_2^m\times D_2\right)
+\left(\amal{D_2^z}{\bz_2^m}{}{D_2}{D_4}\right)
+\left(\amal{\bz_2^p}{\bz_2^m}{}{D_2}{D_4}\right)
\\
&\quad
-\left(D_2\times D_2\right)
-\left(\amal{D_3^p}{\bz_1}{}{\bz_2}{D_{12}}\right)
-\left(\amal{D_4^p}{\bz_2^m}{}{\bz_2}{D_8}\right)
\\
&\quad
-\left(\amal{D_2^p}{D_2^m}{}{D_2}{D_4}\right)
+\left(\amal{D_4}{D_2}{}{D_2}{D_4}\right)
+\left(D_4^z\times D_2\right)
\\
&\quad
-\left(\amal{D_3^p}{D_3}{}{D_2}{D_4}\right)
-\left(\amal{D_4^p}{D_4^z}{}{D_2}{D_4}\right).
\end{aligned}
\]


\begin{thebibliography}{999}

\bibitem{AED}
Z.~Balanov, W.~Krawcewicz, and H.~Steinlein,
\textit{Applied Equivariant Degree},
AIMS Series on Differential Equations \& Dynamical Systems, Vol.~1, American Institute of Mathematical Sciences, 2006.

\bibitem{SURVEY}
Z.~Balanov, W.~Krawcewicz, S.~Rybicki, and H.~Steinlein,
\textit{A short treatise on the equivariant degree theory and its applications},
\emph{J. Fixed Point Theory Appl.}
\textbf{8} (2010), no.~1, 1--74.
\href{https://doi.org/10.1007/s11784-010-0033-9}
{doi:10.1007/s11784-010-0033-9}


\bibitem{Dab1}
M.~Dabkowski, W.~Krawcewicz, Y.~Lv, and H.-P.~Wu,
\textit{Multiple periodic solutions for \(\Gamma\)-symmetric Newtonian systems},
\emph{J. Differential Equations}
\textbf{263} (2017), no.~10, 6684--6730.
\href{https://doi.org/10.1016/j.jde.2017.07.027}
{doi:10.1016/j.jde.2017.07.027}

\bibitem{SY2}
C.~García-Azpeitia, W.~Krawcewicz, S.~Yu, and H.-P.~Wu,
\textit{Subharmonic solutions in reversible difference equations},
\emph{J. Nonlinear Convex Anal.}
\textbf{24} (2023), no.~3, 641--667.
\href{https://yokohamapublishers.jp/online2/opjnca/vol24/p641.html}
{https://yokohamapublishers.jp/online2/opjnca/vol24/p641.html}

\bibitem{Rabinowitz1971}
P.~H.~Rabinowitz,
\textit{Some global results for nonlinear eigenvalue problems},
\emph{J. Funct. Anal.}
\textbf{7} (1971), 487--513.
\href{https://doi.org/10.1016/0022-1236(71)90030-9}
{doi:10.1016/0022-1236(71)90030-9}

\bibitem{SY1}
W.~Krawcewicz, H.-P.~Wu, and S.~Yu,
\textit{Periodic solutions in reversible second-order autonomous systems
with symmetries},
\emph{J. Nonlinear Convex Anal.}
\textbf{18} (2017), no.~8, 1393--1419.
\href{https://www.yokohamapublishers.jp/online2/opjnca/vol18/1393.html}
{https://www.yokohamapublishers.jp/online2/opjnca/vol18/1393.html}

\bibitem{Kur}
K.~Kuratowski,
\textit{Topology},
Vol.~II, Academic Press, New York--London;
PWN--Polish Scientific Publishers, Warsaw, 1968.

\bibitem{Yu3}
S.~Yu,
\textit{Periodic solutions to a ring of identical cells with delay via the
equivariant degree method},
\emph{Int. J. Differ. Equations}
\textbf{2026} (2026), Article ID 2137372.
\href{https://doi.org/10.1155/ijde/2137372}
{doi:10.1155/ijde/2137372}

\bibitem{Podlubny1999}
I.~Podlubny,
\textit{Fractional Differential Equations:
An Introduction to Fractional Derivatives, Fractional Differential Equations,
to Methods of Their Solution and Some of Their Applications},
Mathematics in Science and Engineering, Vol.~198,
Academic Press, San Diego, 1999.

\bibitem{Diethelm2010}
K.~Diethelm,
\textit{The Analysis of Fractional Differential Equations:
An Application-Oriented Exposition Using Differential Operators of Caputo Type},
Lecture Notes in Mathematics, Vol.~2004,
Springer, Berlin--Heidelberg, 2010.
\href{https://doi.org/10.1007/978-3-642-14574-2}
{doi:10.1007/978-3-642-14574-2}

\bibitem{Ferrari2018}
F.~Ferrari,
\textit{Weyl and Marchaud derivatives: A forgotten history},
\emph{Mathematics}
\textbf{6} (2018), no.~1, Article~6.
\href{https://doi.org/10.3390/math6010006}
{doi:10.3390/math6010006}

\bibitem{RoncalStinga2014}
L.~Roncal and P.~R.~Stinga,
\textit{Fractional Laplacian on the torus},
\emph{Commun. Contemp. Math.}
\textbf{18} (2016), no.~3, 1550033.
\href{https://doi.org/10.1142/S0219199715500339}
{doi:10.1142/S0219199715500339}

\bibitem{Sampedro2025}
J.~C.~Sampedro,
\textit{Periodic solutions to nonlocal pseudo-differential equations:
A bifurcation theoretical perspective},
\emph{Nonlinear Anal.}
\textbf{254} (2025), 113746.
\href{https://doi.org/10.1016/j.na.2025.113746}
{doi:10.1016/j.na.2025.113746}

\bibitem{Charkaoui2025}
A.~Charkaoui,
\textit{Temporal periodic solutions for fractional evolution equations
with nonlinear source terms},
\emph{Evol. Equ. Control Theory}
\textbf{14} (2025), no.~6, 1638--1659.
\href{https://doi.org/10.3934/eect.2025047}
{doi:10.3934/eect.2025047}

\bibitem{Haacker2026}
P.-E.~Haacker, R.~I.~Leine, R.~Chaudhary, K.~Diethelm,
A.~Schmidt, and S.~Hashemishahraki,
\textit{Hill-type stability analysis of periodic solutions of
fractional-order differential equations},
\emph{Nonlinear Dyn.}
\textbf{114} (2026), Article~334.
\href{https://doi.org/10.1007/s11071-025-12196-8}
{doi:10.1007/s11071-025-12196-8}

\bibitem{GarciaGhanemKrawcewicz2025}
C.~García-Azpeitia, Z.~Ghanem, and W.~Krawcewicz,
\textit{Global bifurcation in symmetric systems of nonlinear wave equations},
\emph{J. Differential Equations}
\textbf{446} (2025), 113600.
\href{https://doi.org/10.1016/j.jde.2025.113600}
{doi:10.1016/j.jde.2025.113600}

\bibitem{ChenCraneHensley2025}
C.~Chen, C.~Crane, and T.~Hensley,
\textit{Global Hopf bifurcation in symmetric configurations of distributed
delay differential equations},
\emph{Commun. Pure Appl. Anal.} (2025).
\href{https://doi.org/10.3934/cpaa.2025118}
{doi:10.3934/cpaa.2025118}

\bibitem{Izydorek2025}
M.~Izydorek, J.~Janczewska, M.~Starostka, and N.~Waterstraat,
\textit{The equivariant spectral flow and bifurcation for functionals
with symmetries: Part I},
\emph{Math. Ann.}
\textbf{393} (2025), 2187--2226.
\href{https://doi.org/10.1007/s00208-025-03291-7}
{doi:10.1007/s00208-025-03291-7}

\bibitem{Yu2026Random}
S.~Yu,
\textit{Equivariant degree approach to random periodic solutions in second-order stochastic systems with symmetry},
\emph{AIMS Math.}
\textbf{11} (2026), no.~7, 22623--22657.
\href{https://doi.org/10.3934/math.2026914}
{doi:10.3934/math.2026914}

\bibitem{Garcia2019Vortex}
C.~García-Azpeitia,
\textit{Relative periodic solutions of the (n)-vortex problem on the sphere},
\emph{J. Geom. Mech.}
\textbf{11} (2019), no.~3, 427--438.
\href{https://doi.org/10.3934/jgm.2019021}
{doi:10.3934/jgm.2019021}

\bibitem{GarciaBerezovik2019}
I.~Berezovik, C.~García-Azpeitia, and W.~Krawcewicz,
\textit{Symmetries of nonlinear vibrations in tetrahedral molecular configurations},
\emph{Discrete Contin. Dyn. Syst. Ser. B}
\textbf{24} (2019), no.~6, 2473--2491.
\href{https://doi.org/10.3934/dcdsb.2018261}
{doi:10.3934/dcdsb.2018261}

\bibitem{Rybicki1994}
S.~Rybicki,
\textit{A degree for \(S^1\)-equivariant orthogonal maps and its applications
to bifurcation theory},
\emph{Nonlinear Anal.}
\textbf{23} (1994), no.~1, 83--102.
\href{https://doi.org/10.1016/0362-546X(94)90253-4}
{doi:10.1016/0362-546X(94)90253-4}

\bibitem{Golebiewska2018}
A.~Go{\l}\k{e}biewska
\textit{Periodic solutions of asymptotically linear autonomous Hamiltonian
systems with resonance},
\emph{J. Dyn. Differential Equations}
\textbf{30} (2018), 1509--1524.
\href{https://doi.org/10.1007/s10884-017-9608-0}
{doi:10.1007/s10884-017-9608-0}

\bibitem{GolebiewskaRybickiStefaniak2021}
A.~Gołębiewska, S.~Rybicki, and P.~Stefaniak,
\textit{Connected sets of solutions of symmetric elliptic systems},
\emph{Nonlinear Anal.}
\textbf{202} (2021), 112124.
\href{https://doi.org/10.1016/j.na.2020.112124}
{doi:10.1016/j.na.2020.112124}

\bibitem{MarzantowiczPrieto2004}
W.~Marzantowicz and C.~Prieto,
\textit{The unstable equivariant fixed point index and the equivariant degree},
\emph{J. Lond. Math. Soc.}
\textbf{69} (2004), no.~1, 214--230.
\href{https://doi.org/10.1112/S0024610703004721}
{doi:10.1112/S0024610703004721}

            \bibitem{Duan2024}
            Y. Duan, C. Crane, W. Krawcewicz, H. Xiao, {\it Periodic solutions in reversible symmetric second order systems with multiple distributed delays}, Journal of Differential Equations, Volume
401, 282-307, 2024.

\end{thebibliography}
\end{document}